\documentclass[12pt,CJK]{article}
\usepackage{amsfonts,amsmath,amssymb,amsthm,mathrsfs}
\usepackage{latexsym}
\usepackage{graphicx}
\usepackage{indentfirst}
\usepackage{color}
\usepackage{fancyhdr}
\usepackage[colorlinks,linktocpage,linkcolor=blue]{hyperref}

\theoremstyle{plain}
\newtheorem{thm}{Theorem}[section]

\newtheorem{lemma}[thm]{Lemma}
\newtheorem{corollary}[thm]{Corollary}
\newtheorem{remark}[thm]{Remark}
\numberwithin{equation}{section}

\theoremstyle{remark}

\def\Xint#1{\mathchoice
  {\XXint\displaystyle\textstyle{#1}}%
  {\XXint\textstyle\scriptstyle{#1}}%
  {\XXint\scriptstyle\scriptscriptstyle{#1}}%
  {\XXint\scriptscriptstyle\scriptscriptstyle{#1}}%
  \!\int}
\def\XXint#1#2#3{{\setbox0=\hbox{$#1{#2#3}{\int}$}
  \vcenter{\hbox{$#2#3$}}\kern-.5\wd0}}

\def\dashint{\Xint-}

\title{\textbf{Optimal Boundary Estimates for Stokes Systems
in Homogenization Theory }}

\author{ Shu Gu
\thanks{Email: gu@math.fsu.edu.}\\
Department of Mathematics, Florida State University, \\
Tallahassee, FL 32306-4510, USA. \vspace{0.5cm}\\
Qiang Xu
\thanks{Email: xuqiang@math.pku.edu.cn.}\\
School of Mathematical Sciences, Peking University, \\
Beijing, 100871, PR China.
}

\begin{document}
\allowdisplaybreaks
\maketitle
\begin{abstract}
The paper concerns the sharp boundary regularity estimates in homogenization of Dirichlet problem for Stokes systems. We obtain the Lipschitz estimates for velocity term and $L^\infty$ estimate for pressure term, under some reasonable smoothness assumption on
rapidly oscillating periodic coefficients. The approach is based
on convergence rates, originally investigated by S. Armstrong and Z. Shen in \cite{SZ,SZW12}, however the argument developed here does not rely on the Rellich estimates. In this sense, we find a new way to obtain the sharp uniform boundary estimates without imposing the symmetry assumption on coefficients. Additionally, we emphasize that $L^\infty$ estimate for the pressure term does require the $O(\varepsilon^{1/2})$ convergence rate, locally at least,
compared to $O(\varepsilon^\lambda)$ for the velocity term, where $\lambda\in(0,1/2)$.
\end{abstract}

\section{Instruction and main results}
In the paper \cite{SGZWS}, the first author and Z. Shen have systematically established the uniform
estimates for Stokes systems with rapidly oscillating periodic coefficients, such as
the $W^{1,p}$, H\"older estimates, and interior Lipschitz estimates. The compactness argument in \cite{SGZWS}, however, may not be directly applicable in obtaining the boundary Lipschitz estimate for the velocity yet, because of the lack of Green function estimates of Stokes systems with variable coefficients. In this paper, we would like to investigate the sharp boundary estimates of Dirichlet problem, by using the convergence rate estimates instead.

To be precise, we consider the following Stokes systems with the Dirichlet
boundary condition
\begin{equation*}
(\text{D}_\varepsilon)\left\{
\begin{aligned}
\mathcal{L}_\varepsilon(u_\varepsilon) + \nabla p_\varepsilon &= F &\quad &\text{in}~~\Omega, \\
 \text{div} (u_\varepsilon) &= h &\quad&\text{in} ~~\Omega,\\
 u_\varepsilon &= g &\quad&\text{on} ~\partial\Omega,
\end{aligned}\right.
\end{equation*}
with the compatibility condition
\begin{equation}\label{a:5}
\int_\Omega h \  dx = \int_{\partial\Omega} n\cdot g\ dS,
\end{equation}
where $n$ is the unit outward normal to $\partial\Omega$, $\varepsilon>0$ is a small parameter, and the operator $\mathcal{L}_\varepsilon$ is defined by
\begin{equation*}
 \mathcal{L}_\varepsilon = - \text{div}\big[A(x/\varepsilon)\nabla\big]
 = -\frac{\partial}{\partial x_i}
 \Big[a_{ij}^{\alpha\beta}\Big(\frac{x}{\varepsilon}\Big)\frac{\partial}{\partial x_j}\Big].
\end{equation*}
Here $d\geq 2$ and $1 \leq i,j,\alpha,\beta\leq d$, and the summation convention for repeated indices is used throughout. We now assume that the coefficient matrix $A = (a_{ij}^{\alpha\beta})$ is real and satisfies
the uniform ellipticity condition
\begin{equation}\label{a:1}
 \mu |\xi|^2 \leq a_{ij}^{\alpha\beta}(y)\xi_i^\alpha\xi_j^\beta\leq \mu^{-1} |\xi|^2
 \quad \text{for}~y\in\mathbb{R}^d~\text{and}~\xi=(\xi_i^\alpha)\in \mathbb{R}^{d\times d},
 \quad\text{where}~\mu>0;
\end{equation}
and the periodicity condition
\begin{equation}\label{a:2}
A(y+z) = A(y)
\qquad\text{for}~y\in \mathbb{R}^d ~\text{and}~ z\in \mathbb{Z}^d.
\end{equation}
We also impose the smoothness condition, i.e.,
\begin{equation}\label{a:3}
 |A(x)-A(y)| \leq \kappa|x-y|^\tau
 \qquad \text{for}~x,y\in \mathbb{R}^d, \quad \text{where}~\tau\in(0,1).
\end{equation}

The following is the main result of the paper.

\begin{thm}\label{thm:1.0}
Let $\Omega$ be a bounded $C^{1,\tau}$ domain.
Suppose $A$ satisfies $\eqref{a:1}-\eqref{a:3}$.
Given $F\in L^p(\Omega;\mathbb{R}^d)$ and $h\in W^{1,p}(\Omega)$ with
$p>d$, and $g\in C^{1,\eta}(\Omega;\mathbb{R}^d)$ satisfying the compatibility condition $\eqref{a:5}$, where $\eta\in(0,\tau]$,
let $(u_\varepsilon, p_\varepsilon)\in H^1(\Omega;\mathbb{R}^d)
\times L^2(\Omega)$ be a weak solution of the Stokes system
$(\emph{D}_\varepsilon)$.
Then we have the uniform estimate
\begin{equation}\label{pri:1.0}
\begin{aligned}
\big\|\nabla u_\varepsilon\big\|_{L^\infty(\Omega)}
&+ \big\|p_\varepsilon-\dashint_{\Omega}p_\varepsilon\big\|_{L^\infty(\Omega)}
\leq C\bigg\{\|F\|_{L^p(\Omega)} + \|h\|_{W^{1,p}(\Omega)} + \|g\|_{C^{1,\eta}(\partial\Omega)}\bigg\},
\end{aligned}
\end{equation}
where $C$ depends only on $\mu,\kappa,\tau,d,\eta,p$ and $\Omega$.
\end{thm}

We mention that the estimate $\eqref{pri:1.0}$ is sharp even with the $C^\infty$ data and domains. Let us first recap the important development in quantitative homogenization theory, especially in the periodic settings. In the late 1980s, uniform regularity estimates for elliptic systems with Dirichlet boundary conditions was first proved by M. Avellandeda and F. Lin \cite{MAFHL}, where the compactness method was introduced. However, it was not until 2013 that the regularity estimates for the elliptic Neumann boundary problems was solved by C. Kenig, F. Lin and Z. Shen \cite{SZW4}.
Another recent breakthrough was made by S. Armstrong and Z. Shen in \cite{SZ} for the almost-periodic setting, and they developed a new method which based on convergence rates rather than the compactness methods.
We refer the reader to \cite{SACS,SZW21,ZG1} and its reference therein for more details on non-periodic cases. Meanwhile, T. Suslina \cite{TS2,TS} obtained the sharp $O(\varepsilon)$ convergence rates in $L^2(\Omega)$ for elliptic homogenization problems in $C^{1,1}$ domains, while C. Kenig, F. Lin, and Z. Shen \cite{SZW2} figured out the almost sharp one $O(\varepsilon\ln(1/\varepsilon))$ concerned with
Lipschitz domains, and their results have been improved by the second author in \cite{QX2}, recently.
If the reader interests in the boundary estimates, we highly recommend Z. Shen's elegant work \cite{SZW12}. The quantitative homogenization has been extensively
studied, we refer the reader to \cite{SACS,SZ,MAFHL,ABJLGP,GZ,GZS,GZS1,G,SGZWS,SZW20,SZW21,QXS1,ZVVPSE} and their references therein.

For the case of Stokes systems $(\text{D}_\varepsilon)$, the uniform interior estimates and boundary H\"older estimates for the Dirichlet problem have already been established by the first author and Z. Shen \cite{SGZWS}. Now we only focus ourselves on the corresponding boundary estimates. For simplicity, we use the following notation throughout. Let
\begin{equation*}
\begin{aligned}
D_r &= \Big\{(x^\prime,x_d)\in\mathbb{R}^d:|x^\prime|<r ~\text{and}~\psi(x^\prime)<x_d<\psi(x^\prime)+r\Big\},\\
\Delta_r & = \Big\{(x^\prime,x_d)\in\mathbb{R}^d:|x^\prime|<r ~\text{and}~ x_d =\psi(x^\prime)\Big\},
\end{aligned}
\end{equation*}
where $\psi:\mathbb{R}^{d-1}\to \mathbb{R}$ is a $C^{1,\tau}$ function for some $\tau\in(0,1)$ with
$\psi(0) = 0$ and $\|\nabla\psi\|_{C^{0,\tau}(\mathbb{R}^{d-1})}\leq M$.

\begin{thm}[Local boundary estimates]\label{thm:1.1}
Assume the same conditions as in Theorem $\ref{thm:1.0}$. Let $(u_\varepsilon, p_\varepsilon)\in H^1(D_5;\mathbb{R}^d)
\times L^2(D_5)$ be a weak solution of
$\mathcal{L}_\varepsilon(u_\varepsilon) + \nabla p_\varepsilon = F$
and $\emph{div}(u_\varepsilon) = h$ in
$D_5$ and $u_\varepsilon = g$ on $\Delta_5$,
where $F\in L^p(D_5;\mathbb{R}^d)$ and $h\in W^{1,p}(D_5)$ with
$p>d$, and $g\in C^{1,\eta}(\Delta_5;\mathbb{R}^d)$ with $g(0) = 0$, where $\eta\in(0,\tau]$.
Then there holds
\begin{equation}\label{pri:1.1}
\begin{aligned}
\Big(\dashint_{D_r} |\nabla u_\varepsilon|^2dx\Big)^{\frac{1}{2}}
+ \Big|\dashint_{D_r}p_\varepsilon - \dashint_{D_1}p_\varepsilon dx\Big|
&\leq C\bigg\{\Big(\dashint_{D_2} |u_\varepsilon|^2dx\Big)^{\frac{1}{2}} + \Big(\dashint_{D_2} |F|^pdx\Big)^{\frac{1}{p}}\\
& \qquad+ \Big(\dashint_{D_2} |\nabla h|^pdx\Big)^{\frac{1}{p}}
+ \|h\|_{L^\infty(D_2)}
+ \|g\|_{C^{1,\eta}(\Delta_2)}\bigg\}
\end{aligned}
\end{equation}
for any $0<r<1/4$,
where $C$ depends only on $\mu,\kappa,\tau,d,M,\eta$ and $p$.
\end{thm}

The scaling-invariant estimate $\eqref{pri:1.1}$ ought to be regarded as a
Lipschitz estimate for the velocity $u_\varepsilon$ and $L^\infty$ estimate for the pressure $p_\varepsilon$, since it is not hard to bound the quantity
\begin{equation*}
|\nabla u_\varepsilon(0)| + \big|p_\varepsilon(0)-\dashint_{D_{1}} p_\varepsilon\big|
\end{equation*}
by the right-hand side of $\eqref{pri:1.1}$ due to
the blow-up argument, where
$0\in\partial\Omega$.
Before explaining our tactics, let us review the ideas developed in \cite{SZ,SZW12}.
For elliptic operator $\mathcal{L}_\varepsilon$ with Dirichlet or Neumann boundary conditions,
they found that the quantity
\begin{equation*}
\Big(\dashint_{D_r} |\nabla u_\varepsilon|^2 dx\Big)^{1/2}
\end{equation*}
could be bounded for any scale $r$ uniformly down to $\varepsilon$,
provided the coefficients own some repeated self-similar structure, for example,
which may be periodic or almost-periodic, even random in stationary and ergodic setting.
This result indeed came from the so-called Campanato iteration.
However, it requires an effective control of the error in homogenization as a precondition. Take the periodic homogenization as an example,
Z. Shen \cite{SZW12} construct a function $v$ such that $\mathcal{L}_0(v) = F$ in $D_r$ with the same (Dirichlet or Neumann) data on $\Delta_r$ as $u_\varepsilon$, where $\mathcal{L}_0$ is the homogenized (effective) operator of $\mathcal{L}_\varepsilon$, and
\begin{equation}\label{f:0.1}
\Big(\dashint_{D_r} |u_\varepsilon -v|^2 dx\Big)^{1/2}
\leq C\Big(\frac{\varepsilon}{r}\Big)^{1/2}
\bigg\{\Big(\dashint_{D_{2r}}|u_\varepsilon|^2dx\Big)^{\frac{1}{2}}+\text{terms involving given data}\bigg\}.
\end{equation}
It is a scaling-invariant estimate, so it suffices to consider the case of $r=1$, and this is
exactly where the convergence rates
$\|\mathcal{L}_\varepsilon^{-1} - \mathcal{L}_0^{-1}\|_{L^2(\Omega)\to L^2(\Omega)} = O(\varepsilon^{1/2})$ play a role. We mention that this estimate is not as simple as it appears,
even for elliptic systems in a bounded $C^{1,\tau}$ domain. The hard part is to
control the second order derivative lacking the smoothness assumptions on coefficients or domains.
However, if the operator $\mathcal{L}_\varepsilon$ satisfies a symmetry condition, i.e., $A=A^*$,
then the Rellich estimate became a powerful tool, which makes it possible to use the nontangential maximal function to control the boundary behavior of the solution. We mention that the estimate
$\eqref{f:0.1}$ is established in Lipschitz domains without any smoothness assumption on coefficients, but it relies on the symmetry condition.
We refer the reader to \cite{SZ,SZW12} for the original thinking.

In the paper, although the main idea is similar as that in \cite{SZ,SZW12},
the innovation clearly reflects in two aspects:
the estimate $\eqref{pri:1.1}$ does not depend on any symmetry condition, which can be extended
to the Neumann boundary problems without real difficulties;
using convergence rates on pressure term recovers its uniform boundary $L^\infty$ estimate, which shows an approach unlike the iteration arguments applied to the velocity term.

In the following paragraphs, we will outline our strategy related to the estimate $\eqref{pri:1.1}$. As the second author has found in $\cite{QX2}$, originally motivated by Z. Shen in \cite{SZW12},
the order of the convergence rates is determined by the so-called ``layer'' and ``co-layer''
type estimates for the homogenized boundary problems. If $\Omega$ is a bounded $C^1$ domain, the
interior Schauder estimate combining with the global H\"older estimate leads to
\begin{equation}\label{f:0.2}
\big\|u_\varepsilon - u_0\big\|_{L^2(\Omega)}
\leq C\varepsilon^{\sigma-\frac{1}{2}}\Big\{\|F\|_{L^p(\Omega)}+\|h\|_{W^{1,p}(\Omega)}
+\|g\|_{C^{0,1}(\partial\Omega)}\Big\}
\end{equation}
for any $\sigma\in(1/2,1)$. Hence, under the same assumptions as in Theorem $\ref{thm:1.1}$, we construct the solution $(u_0,p_0)$ satisfying  $\mathcal{L}_0(u_0)+\nabla p_0 = F$, and $\text{div}(u_0) = h$ in $D_2$,
and $u_0 = u_\varepsilon$ on $\partial D_2$, it follows from the estimate $\eqref{f:0.2}$ that
\begin{equation}\label{f:0.3}
\big\|u_\varepsilon - u_0\big\|_{L^2(D_1)}
\leq C\varepsilon^{\sigma-\frac{1}{2}}\Big\{\|\nabla u_\varepsilon\|_{L^{\infty}(\partial D_2\setminus\Delta_2)}+\text{terms involve given data}\Big\}.
\end{equation}
Obviously, the challenging task is to estimate the quantity $\|\nabla u_\varepsilon\|_{L^{\infty}(\partial D_2\setminus\Delta_2)}$.
Although we can not count on bounding it uniformly, it is possible to derive a nonuniform estimate
for the first term of the right-hand side of $\eqref{f:0.3}$, and we hope that its ``diverging order'' will be smaller than $\sigma-\frac{1}{2}$. In fact, the local Lipschitz estimate together
with the uniform global H\"older estimate gives us
\begin{equation}\label{f:0.4}
\|\nabla u_\varepsilon\|_{L^{\infty}(\partial D_2\setminus\Delta_2)}
\leq C\varepsilon^{\sigma-1}\Big\{\|u_\varepsilon\|_{L^2(D_4)}+\text{terms involve given data}\Big\}.
\end{equation}
Combining the estimates $\eqref{f:0.3}$ and $\eqref{f:0.4}$, we arrive at
\begin{equation*}
\|u_\varepsilon - u_0\|_{L^2(D_1)}
\leq C\varepsilon^{\lambda}\Big\{\|u_\varepsilon\|_{L^2(D_4)}+\text{terms involve given data}\Big\},
\end{equation*}
where $\lambda=2\sigma-\frac{3}{2}$.
It is clear to see that $\lambda\in(0,1/2)$ whenever $(3/4)<\sigma<1$.
We emphasize that we just require a positive power $\lambda$ to proceed the Campanato iteration for
Lipschitz estimates on velocity $u_\varepsilon$.

Compared to the methods developed in \cite{SZW12}, ours involve with  the microcosmic information on coefficients, which may be regarded as the price of sacrificing the symmetry conditions.
The lucky thing is that the requirements of additional smoothness assumptions are not beyond
those in common cases. Another remark is that we essentially use the uniform H\"older
estimates in $\eqref{f:0.4}$, which have already been established in \cite{SGZWS} by the well known
compactness method. In this sense,
the arguments developed here actually blend the compactness methods, convergence rates coupled with the Campanato iteration. Here the iteration argument
(see Lemma $\ref{lemma:5.3}$) belongs to Z. Shen, who notably simplified the proof in \cite{SZW12}.
If the model is replaced by elliptic systems, we may even obtain sharp H\"older estimates for Dirichlet problems, that means it suffices to bound the quantity $\|u_\varepsilon\|_{C^{0,\eta^\prime}(\partial D_2\setminus\Delta_2)}$, where $\eta^\prime\in(0,1)$, and it will be done by the uniform H\"older estimates without the blow-up arguments.
However,
the case of Neumann boundary follows the same way as we addressed here,
since one has to estimate the quantity $\|\nabla u_\varepsilon\|_{L^{\infty}(\partial D_2\setminus\Delta_2)}$, as well. To some extend, this suggests that the Neumann problem will be
more complicated than the Dirichlet one. We will show the Stokes systems with the Neumann boundary conditions in a separate work.
The paragraph ends here by mentioning that the idea of a nonuniform estimate has already been used in the study of elliptic systems with lower order terms by the second author (see \cite{QXS1}).

We now turn to show how to bound the pressure $p_\varepsilon$ in $L^\infty$-norm, which reflects the other innovation of the paper. Note that Lipschitz estimates for $u_\varepsilon$ do not
simply implies the $L^\infty$ estimate for $p_\varepsilon$, since $p_\varepsilon$ is related to $\nabla u_\varepsilon$ by a singular integral. We find that due to the
local Schauder estimates for $\mathcal{L}_0$, it is not hard to derive
\begin{equation}
 \Big(\dashint_{D_r} \big|p_0 - \dashint_{D_s} p_0\big|^2 dx\Big)^{\frac{1}{2}}
 \leq Cs^\rho\bigg\{\|u_0\|_{L^2(D_2)}
 + \text{terms involve given data}\bigg\}
\end{equation}
for any $0<r\leq s\leq 1$, where $ 0<\rho<\min\{\eta,1-d/p\}$. In this form, the above estimate is quite similar to the desired estimate for $p_\varepsilon$. Hence, roughly speaking, the idea is  to transfer the
corresponding estimate for $p_\varepsilon$ to a similar one for $p_0$ by using
the convergent relationship between $p_\varepsilon$ and $p_0$.
However, the first intractable problem is
that $p_\varepsilon$ just weakly converges to $p_0$
in $L^2(\Omega)$,  and we can not expect any precise control of the error. Fortunately, it is known that  $p_\varepsilon$ subtracting its first order approximating corrector strongly converges in $L^2(\Omega)/\mathbb{R}$, where $L^2(\Omega)/\mathbb{R}$ denotes the quotient space of $L^2(\Omega)$ with respect to the relation $u\sim v \Leftrightarrow u-v\in\mathbb{R}$. As will be shown in Section 5, we indeed obtain
\begin{equation}\label{f:0.5}
\begin{aligned}
\|p_\varepsilon-p_0
&-\pi(\cdot/\varepsilon)\psi_{2\varepsilon}
\nabla u_0\|_{L^2(D_1)/\mathbb{R}}\\
&\leq C\varepsilon^{\lambda}\bigg\{
\|u_\varepsilon\|_{L^2(D_4)}
+ \text{terms involve given data}\bigg\},
\end{aligned}
\end{equation}
where $(\chi,\pi)$ is the corrector associated with $(\text{D}_\varepsilon)$, defined in Subsection $\ref{subsec:2.3}$,  and
$\psi_{2\varepsilon}$ is a smooth cut-off function whose expression will be given in Subsection $\ref{subsec:2.1}$.
Although the estimate $\eqref{f:0.5}$ provides us an accurate way to control the rate of the convergence, the control, in fact, is required in each scale $r$ from $\varepsilon$ to $1$.
That means, for example, to bound the quantity
$\big|\dashint_{D_r}p_\varepsilon - \dashint_{D_1}p_\varepsilon dx\big|$ for any $\varepsilon\leq r<(1/4)$, it suffices to consider the quantity
\begin{equation*}
\sum_{j=1}^k \Big|\dashint_{D_{2^{-j}}} p_\varepsilon - \dashint_{D_{2^{-j+1}}} p_\varepsilon dx\Big|
\end{equation*}
where $2^{-k-1}\leq r<2^{-k}$. Hence we need the scaling-invariance version of $\eqref{f:0.5}$ and it will be established in Lemma $\ref{lemma:6.2}$. Besides, in the calculation, we also need to estimate the  quantity
\begin{equation*}
\frac{1}{|D_{2^{-j}}|^{\frac{1}{2}}}\bigg\{\big\|\nabla u_0\big\|_{L^2(D_{2^{-j}}\setminus\Sigma_{4\varepsilon})}
+ \varepsilon\big\|\nabla^2 u_0\big\|_{L^2(D_{2^{-j}}\cap\Sigma_{4\varepsilon})}\bigg\},
\end{equation*}
where $\Sigma_{4\varepsilon} = \{x\in\Omega:\text{dist}(x,\partial\Omega)>4\varepsilon\}$, and it will be
controlled by
\begin{equation*}
C(\varepsilon 2^j)^{\frac{1}{2}}\bigg\{
\|u_0\|_{L^2(D_2)}+\text{terms involve given data}\bigg\}.
\end{equation*}
The concrete statement can be found in Lemma $\ref{lemma:3.5}$.
The above expression also reveals that the convergence rate related to pressure term must reach $O(\varepsilon^{1/2})$ in the local sense,
and it guarantees that
$\sum_{j=1}^{k}(\varepsilon 2^j)^{\frac{1}{2}}$ is convergent whenever $0<\varepsilon<2^{-k}$.
Compared to $O(\varepsilon^\lambda)$ rates for the velocity term, where $\lambda\in(0,1/2)$,
it seems to be an evidence that the $L^\infty$ estimate for the pressure term is harder than
the Lipschitz estimate for the velocity term. In the last step of the proof, we use the fact
that $u_\varepsilon$ strongly converges to $u_0$ in $L^2(D_2)$ due to the homogenization theory.
We finally remark that the case of $0<r\leq \varepsilon$ always follows from the blow-up argument.

The paper is organized as follows. Section 2 is divided into four subsections, including notation,
 smoothing operator and its properties, corrector and its properties, and the classical regularity theory. Section 3 is devoted to study the rate of convergence. We show the uniform Lipschitz estimates on velocity $u_\varepsilon$ in Section 4, and the $L^\infty$ estimates on pressure $p_\varepsilon$ in Section 5.

\section{Preliminaries}
\subsection{Notation in the paper}\label{subsec:2.1}
We first introduce some notation that will be used in the following sections.
\begin{itemize}
  \item $\nabla v = (\nabla_1 v, \cdots, \nabla_d v)$ is the gradient of $v$, where
  $\nabla_i v = \partial v /\partial x_i$ denotes the $i^{\text{th}}$ derivative of $v$. \\
  $\nabla^2 v = (\nabla^2_{ij} v)_{d\times d}$  denotes the Hessian matrix of $v$, where
  $\nabla^2_{ij} v = \frac{\partial^2 v}{\partial x_i\partial x_j}$. \\
  $\text{div}(\textbf{v})=\sum_{i=1}^d \nabla_i v_i$ denotes the divergence of $\textbf{v}$, where
  $\textbf{v} = (v_1,\cdots,v_d)$ is a vector-valued function.
  \item $L^2(\Omega)/\mathbb{R}=\{f\in L^2(\Omega):\int_\Omega f(x)\,dx = 0\}$,
  and $\|f\|_{L^2(\Omega)/\mathbb{R}} = \inf\limits_{c\in\mathbb{R}}\|f-c\|_{L^2(\Omega)}$.
  \item $\delta(x) = \text{dist}(x,\partial\Omega)$ denotes the distance function for $x\in\Omega$, and
  we set $\delta(x) = 0$ if $x\in\mathbb{R}^d\setminus\Omega$.
  \item $S_r = \{x\in\Omega:\text{dist}(x,\partial\Omega) = r\}$ denotes the level set.
  \item $\Omega\setminus\Sigma_r$ denotes the boundary layer with thickness $r>0$, where
  $\Sigma_r = \{x\in\Omega:\text{dist}(x,\partial\Omega)>r\}$.
  \item Let $B=B(x,r)=B_r(x)$, and $kB=B(x,kr)$ denote the concentric balls as $k>0$ varies.
  \item Let $\psi_r$ denote the cut-off function associated with $\Sigma_r$, such that
  \begin{equation}\label{def:2.5}
   \psi_{r} = 1 \quad\text{in}~\Sigma_{2r}, \qquad
   \psi_{r} = 0 \quad\text{outside}~\Sigma_r, \qquad\text{and}\quad |\nabla\psi_r|\leq C/r.
  \end{equation}
\end{itemize}

Throughout the paper, the constant $C$ never depends on $\varepsilon$. Finally
we mention that we shall make a little effort to distinguish vector-valued functions or
function spaces from their real-valued counterparts, and they will be clear from the context.

\subsection{Smoothing operator and its properties}\label{subsection:2.1}

We first state the definition and properties of the smoothing operator $S_\varepsilon$. Detailed proof may be found in \cite{SZW12}.
We fix $\zeta\in C_0^\infty(B(0,1/2))$, and $\int_{\mathbb{R}^d}\zeta(x)dx = 1$ and denote $\zeta_\varepsilon$ by $\zeta_\varepsilon=\varepsilon^{-d}\zeta(x/\varepsilon)$. The smoothing operator $S_\varepsilon$ is defined by
\begin{equation}
S_\varepsilon(f)(x) = f*\zeta_\varepsilon(x) = \int_{\mathbb{R}^d} f(x-y)\zeta_\varepsilon(y) dy,
\end{equation}

\begin{lemma}
Let $f\in L^p(\mathbb{R}^d)$ for some $1\leq p<\infty$. Then for any $\rho\in L_{per}^p(\mathbb{R}^d)$,
\begin{equation}\label{pri:2.7}
\big\|\rho(\cdot/\varepsilon)S_\varepsilon(f)\big\|_{L^p(\mathbb{R}^d)}
\leq C\big\|\rho\big\|_{L^p(Y)}\big\|f\big\|_{L^p(\mathbb{R}^d)},
\end{equation}
where $C$ depends only on $d$.
\end{lemma}

\begin{lemma}\label{lemma:2.4}
Let $f\in W^{1,p}(\mathbb{R}^d)$ for some $1<p<\infty$. Then we have
\begin{equation}\label{pri:2.8}
\big\|S_\varepsilon(f)-f\big\|_{L^p(\mathbb{R}^d)}
\leq C\varepsilon\big\|\nabla f\big\|_{L^p(\mathbb{R}^d)},
\end{equation}
and further obtain
\begin{equation}\label{pri:2.9}
\big\|S_\varepsilon(f)\big\|_{L^2(\mathbb{R}^d)}\leq C\varepsilon^{-1/2}\big\|f\big\|_{L^q(\mathbb{R}^d)}
\quad\text{and}\quad
\big\|S_\varepsilon(f)-f\big\|_{L^2(\mathbb{R}^d)}
\leq C\varepsilon^{1/2}\big\|\nabla f\big\|_{L^q(\mathbb{R}^d)},
\end{equation}
where $q = 2d/(d+1)$, and $C$ depends only on $d$.
\end{lemma}

\subsection{Correctors and its properties}\label{subsec:2.3}
We will use this subsection to introduce the correctors and dual correctors of Stokes systems in the homogenization theory. Details can be found in the literatures such as \cite{SGZWS,G,G1,G2,QX3}.

Let $Y = [0,1)^d\backsimeq \mathbb{R}^d/\mathbb{Z}^d$.
We define the correctors
$(\chi_k^{\beta\gamma},\pi_k^\gamma)\in H^1_{per}(Y;\mathbb{R}^d)\times L_{per}^2(Y)$
associated with the Stokes
system $(D_\varepsilon)$ by the following cell problem:
\begin{equation}\label{pde:2.1}
\left\{\begin{aligned}
\mathcal{L}_1(\chi_k^\gamma+P_k^\gamma) + \nabla \pi_k^\gamma & = 0 \qquad \text{in}~~\mathbb{R}^d,\\
\text{div}(\chi_k^\gamma) &  = 0  \qquad \text{in}~~\mathbb{R}^d,\\
\int_Y \chi_k^\gamma\  dy &= 0, \qquad k,\gamma = 1,\cdots,d, 
\end{aligned}\right.
\end{equation}
where $P_k^\gamma = y_k e^\gamma = y_k(0,\cdots,1,\cdots,0)$ with 1 in the $\gamma^{\text{th}}$ position,
and $\chi_k^\gamma = (\chi_k^{1\gamma},\cdots,\chi_k^{d\gamma})$. It follows
from \cite[Theorem 2.1]{SGZWS} that
\begin{equation}\label{pri:2.4}
 \|\nabla \chi_k^{\gamma}\|_{L^2(Y)} + \|\pi_k^\gamma\|_{L^2(Y)} \leq C,
\end{equation}
where $C$ depends only on $\mu$ and $d$.
Then the homogenized operator is given by $\mathcal{L}_0 = -\text{div}(\widehat{A}\nabla)$,
where $\widehat{A} = (\hat{a}_{ij}^{\alpha\beta})$ and
\begin{equation}\label{eq:2.1}
 \hat{a}_{ij}^{\alpha\beta} = \int_Y \Big[a_{ij}^{\alpha\beta}
 + a_{ik}^{\alpha\gamma}\frac{\partial}{\partial y_k}\big(\chi_j^{\gamma\beta}\big) \Big]dy
\end{equation}
By the homogenization theory proved in \cite{SGZWS,ABJLGP}, we know that
$$
u_\varepsilon \rightharpoonup u_0 \text{ weakly in } H^1(\Omega;\mathbb{R}^d)\quad \text{ and }\quad p_\varepsilon-\dashint_\Omega p_\varepsilon \rightharpoonup p_0-\dashint_\Omega p_0 \text{ weakly in }L^2(\Omega),
$$
where $(u_0,p_0)$ is the weak solution of the homogenized problem,
\begin{equation*}
(\text{D}_0)\left\{
\begin{aligned}
\mathcal{L}_0(u_0) + \nabla p_0 &= F &\quad &\text{in}~~\Omega, \\
 \text{div} (u_0) &= h &\quad&\text{in} ~~\Omega,\\
 u_0 &= g &\quad&\text{on} ~\partial\Omega.
\end{aligned}\right.
\end{equation*}
We define
\begin{equation*}
 b_{ik}^{\alpha\gamma}(y) = \hat{a}_{ik}^{\alpha\gamma}
 - a_{ik}^{\alpha\gamma}(y) - a_{ij}^{\alpha\beta}(y)\frac{\partial\chi_k^{\beta\gamma}}{\partial y_j}(y), \quad y=\frac{x}{\varepsilon},
\end{equation*}
and it is obvious that $\int_Y b_{ik}^{\alpha\gamma} dy=0$, and $\nabla_i b_{ik}^{\alpha\gamma} = -\nabla_\alpha\pi_k^\gamma$.
As usual, the following lemma provides the definition and the properties of $(E_{jik}^{\alpha\gamma}, q_{ik}^\gamma)$ of Stokes systems, we refer \cite{G,G1,G2,QX3} for more details.
\begin{lemma}\label{lemma:2.3}
There exist $E_{jik}^{\alpha\gamma}\in H^1_{per}(Y)$ and
$q_{ik}^\gamma\in H_{per}^1(Y)$ such that
\begin{equation}\label{eq:2.2}
 b_{ik}^{\alpha\gamma} = \nabla_j E_{jik}^{\alpha\gamma} - \nabla_\alpha q_{ik}^{\gamma},
 \qquad
 E_{jik}^{\alpha\gamma} = - E_{ijk}^{\alpha\gamma},
 \qquad \text{and}\qquad
 \nabla_i q_{ik}^\gamma = \pi_k^\gamma,
\end{equation}
and $E_{jik}^{\alpha\gamma}$ and $q_{ik}^{\gamma}$ admit the priori estimate
\begin{equation}\label{pri:2.3}
\|E_{jik}^{\alpha\gamma}\|_{L^2(Y)} + \|q_{ik}^\gamma\|_{L^2(Y)}\leq C,
\end{equation}
where $C$ depends only on $\mu$ and $d$.
\end{lemma}

\begin{lemma}
Let $\big\{E_{jik}^{\alpha\gamma},q_{ik}^\gamma\big\}$ be given
in Lemma $\ref{lemma:2.3}$. Additionally, if the correctors $\chi_k=(\chi_k^{\beta\gamma})$ with $k=1,\cdots,d$
are H\"older continuous, then $E_{jik}^{\alpha\gamma},q_{ik}^\gamma\in L^\infty(Y)$.
\end{lemma}

\begin{proof}
See \cite[Lemma 5.1]{QX3}
\end{proof}

\begin{remark}\label{remark:3.1}
\emph{If $\pi\in L_{per}^2(Y)$,
there exists $V\in H^1_{per}(Y;\mathbb{R}^d)$ such that
$\text{div}_y(V)= \pi(y) -\widehat{\pi}$ in $\mathbb{R}^d$. Moreover
if $\pi\in L^\infty(\mathbb{R}^d)$, then we have $\|V\|_{L^\infty(\mathbb{R}^d)}\leq C(\mu,\tau,\kappa,d)$. The proof is similar to that in \cite[Lemma 5.1]{QX3}.}
\end{remark}

\subsection{Classical regularity theory}
In this section, we recall some classical results, including the local and global versions,
which mainly come from \cite{MGMG,GPGCG,OAL} and the reference therein.
To avoid confusion, we use the notation
$\mathcal{L}(u) = \text{div}\big[A(x)\nabla u\big]$ to denote the operator with variable coefficients that does not depend on $\varepsilon$.

\begin{thm}[Global H\"older estimates]
Let $\Omega$ be a bounded $C^{1}$ domain.
If $A\in \emph{VMO}$ satisfies $\eqref{a:1}$. Given $F\in L^{p}(\Omega;\mathbb{R}^d)$, $h\in L^{2p}(\Omega)$ with $2p>d$ and
$g\in C^{0,1}(\partial\Omega;\mathbb{R}^d)$ satisfying the compatibility condition $\eqref{a:5}$, let $(u,\varrho)\in H^1(\Omega;\mathbb{R}^d)\times L^2(\Omega)$ be the weak  solution to the
Dirichlet problem:
$\mathcal{L}(u)+\nabla\varrho = F$ and $\emph{div}(u) = h$ in $\Omega$ and
$u=g$ on $\partial\Omega$. Then for any $0<\sigma<1$, we have
  \begin{equation}\label{pri:2.16}
  \|u\|_{C^{0,\sigma}(\Omega)} \leq C\Big\{\|F\|_{L^p(\Omega)} + \|h\|_{L^{2p}(\Omega)}
  +\|g\|_{C^{0,1}(\partial\Omega)}\Big\},
  \end{equation}
  where $C$ depends on $\mu,\omega,d,p,\sigma$ and $\Omega$.
\end{thm}

\begin{thm}[Local estimates]
Let $\Omega$ be a $C^{1,\tau}$ domain. Assume that $A$ satisfies $\eqref{a:1}$ and
$\eqref{a:3}$. Let $(u,\varrho)\in H^1(D_5;\mathbb{R}^d)\times L^2(D_5)$ be the weak solution
to $\mathcal{L}(u)+\nabla \varrho = F$ and $\emph{div}(u) = h$ in $D_5$,
and $u=g$ on $\Delta_5$,
where $F\in L^p(D_5;\mathbb{R}^d)$ and $h\in W^{1,p}(D_5)$ with $p>d$, and
$g\in C^{1,\eta}(\Delta_5;\mathbb{R}^d)$ with $g(0)=0$.
Then for any $0<\rho\leq \min\{\eta,1-d/p\}$, we have the following estimate:
\begin{itemize}
  \item [\emph{(1)}] Interior Schauder estimate, for any $B_{2r}\subset D_5$,
  there holds
  \begin{equation}\label{pri:2.10}
  \begin{aligned}
  \big[\nabla u\big]_{C^{0,\rho}(B_r)} + \big[\varrho\big]_{C^{0,\rho}(B_r)}
  &\leq Cr^{-\rho}\bigg\{\frac{1}{r}\Big(\dashint_{B_{2r}}|u-c|^2dx\Big)^{\frac{1}{2}}\\
  &+r\Big(\dashint_{B_{2r}}|F|^pdx\Big)^{\frac{1}{p}}
  +r\Big(\dashint_{B_{2r}}|\nabla h|^pdx\Big)^{\frac{1}{p}}
  +\|h\|_{L^\infty(B_{2r})}
  \bigg\}
  \end{aligned}
  \end{equation}
  for any $c\in\mathbb{R}^d$, where $C$ depends only on $\mu,\tau,\kappa,d,p$ and $\rho$.
  \item [\emph{(2)}] Boundary Schauder estimate, for any $D_{2r}\subset D_5$ and $c\in\mathbb{R}^d$, we have
  \begin{equation}\label{pri:2.18}
  \begin{aligned}
  \big[\nabla u\big]_{C^{0,\rho}(D_r)} + \big[\varrho\big]_{C^{0,\rho}(D_r)}
  &\leq Cr^{-\rho}\bigg\{\frac{1}{r}\Big(\dashint_{D_{2r}}|u|^2dx\Big)^{\frac{1}{2}}+r\Big(\dashint_{D_{2r}}|F|^pdx\Big)^{\frac{1}{p}}+r\Big(\dashint_{D_{2r}}|\nabla h|^pdx\Big)^{\frac{1}{p}}\\
  &\quad+\|h\|_{L^\infty(D_{2r})}
  +\|\nabla g\|_{L^\infty(D_{2r})}
  +r^{\eta}[\nabla g]_{C^{0,\eta}(\Delta_{2r})}
  \bigg\},
  \end{aligned}
  \end{equation}
  where $C$ depends on $\mu,\tau,\kappa,d,p,\eta,M$ and $\rho$.
    \item [\emph{(3)}] Boundary Lipschitz estimate, for any $0<r\leq R\leq 1$, we have
  \begin{equation}\label{pri:2.17}
  \begin{aligned}
  \big\|\nabla u\big\|_{L^{\infty}(D_r)} + \big\|\varrho-&\dashint_{D_{R}}\varrho\big\|_{L^{\infty}(D_r)}
  \leq C\bigg\{\frac{1}{R}\Big(\dashint_{D_{2R}}|u|^2dx\Big)^{\frac{1}{2}}+R\Big(\dashint_{D_{2R}}|F|^pdx\Big)^{\frac{1}{p}}\\
  &+R\Big(\dashint_{D_{2R}}|\nabla h|^pdx\Big)^{\frac{1}{p}}+\|h\|_{L^\infty(D_{2R})}
  +\|\nabla g\|_{L^\infty(D_{2R})}
  +R^{\eta}[\nabla g]_{C^{0,\eta}(\Delta_{2R})}
  \bigg\}
  \end{aligned}
  \end{equation}
  where $C$ depends on $\mu,\tau,\kappa,d,p,\eta,M$ and $\rho$.
\end{itemize}
\end{thm}

\section{Convergence rates}

\begin{thm}\label{thm:3.3}
Let $\Omega$ be a $C^{1}$ domain. Suppose that $A$ satisfies $\eqref{a:1}-\eqref{a:2}$.
Given $F\in L^p(\Omega;\mathbb{R}^d)$ and $h\in W^{1,p}(\Omega)$ with $p>d$,
and $g\in C^{0,1}(\partial\Omega;\mathbb{R}^d)$ satisfying the compatibility condition $\eqref{a:5}$, and we
assume that $(u_\varepsilon,p_\varepsilon)$, $(u_0,p_0)$ in
$H^1(\Omega;\mathbb{R}^d)\times L^2(\Omega)$ are weak solutions to Stokes systems $(\text{D}_\varepsilon)$, $(\text{D}_0)$, respectively. Then for any $(1/2)< \sigma<1$,
we have
\begin{equation}\label{pri:3.13}
\begin{aligned}
\|u_\varepsilon-u_0\|_{L^2(\Omega)} +
\|p_\varepsilon-p_0
 -\pi(\cdot/\varepsilon)S_\varepsilon(\psi_{2\varepsilon}
\nabla u_0)\|_{L^2(\Omega)/\mathbb{R}}\leq C\varepsilon^{\sigma-\frac{1}{2}}\Big\{\|F\|_{L^p(\Omega)}
+ \|h\|_{W^{1,p}(\Omega)}
+\|g\|_{C^{0,1}(\partial\Omega)}\Big\}
\end{aligned}
\end{equation}
where $C$ depends only on $\mu,d,\sigma,p$ and $\Omega$.
\end{thm}

\begin{lemma}\label{lemma:3.4}
Assume the same conditions as in Theorem $\ref{thm:3.3}$. Let $(u_0,p_0)\in H^1(\Omega;\mathbb{R}^d)\times L^2(\Omega)$
be the weak solution to $(\text{D}_0)$. Then for any $\sigma\in(1/2,1)$, we have
\begin{equation}\label{pri:3.11}
\|\nabla u_0\|_{L^2(\Omega\setminus\Sigma_{p_1\varepsilon})}
\leq C\varepsilon^{\sigma-\frac{1}{2}}\Big\{
\|F\|_{L^p(\Omega)}+\|h\|_{W^{1,p}(\Omega)}+\|g\|_{C^{0,1}(\partial\Omega)}\Big\},
\end{equation}
and
\begin{equation}\label{pri:3.12}
\|\nabla^2u_0\|_{L^2(\Sigma_{p_2\varepsilon})} \leq C\varepsilon^{\sigma-\frac{3}{2}}
\Big\{\|F\|_{L^p(\Omega)}+\|h\|_{W^{1,p}(\Omega)}+\|g\|_{C^{0,1}(\partial\Omega)}\Big\},
\end{equation}
where $p_1,p_2>0$ are fixed real number, and $C$ depends on $\mu,d,p_1,p_2,\sigma,p$ and $\Omega$.
\end{lemma}

\begin{proof}
We first handle the estimate $\eqref{pri:3.11}$. It is convenient to assume
$\|F\|_{L^p(\Omega)}+\|h\|_{W^{1,p}(\Omega)}+\|g\|_{C^{0,1}(\partial\Omega)} =1$.
For any $x\in\Omega$, let $r= \delta(x)$.
It follows from the interior Schauder estimates $\eqref{pri:2.10}$ that
\begin{equation}\label{f:3.2}
\begin{aligned}
 |\nabla u_0(x)|
 &\leq C\bigg\{\frac{1}{r}\Big(\dashint_{B(x,r)}\big|u_0(y) - \dashint_{B(x,r)}u_0\big|^2dy\Big)^{\frac{1}{2}} +r\Big(\dashint_{B(x,r)}\big|F\big|^pdy\Big)^{\frac{1}{p}}\\
 &\qquad+r\Big(\dashint_{B(x,r)}\big|\nabla h\big|^pdy\Big)^{\frac{1}{p}}
 +\|h\|_{L^\infty(B(x,r))}\bigg\}\\
 &\leq Cr^{\sigma-1}\big[u_0\big]_{C^{0,\sigma}(\Omega)}
 + C\Big\{\|F\|_{L^p(\Omega)} + \|h\|_{W^{1,p}(\Omega)}\Big\}\leq Cr^{\sigma-1},
\end{aligned}
\end{equation}
where the global H\"older estimate $\eqref{pri:2.16}$ was used in the last inequality. We mention that the range of the H\"older exponent in $\eqref{pri:2.16}$ is $(0,1)$. Then it follows from the co-area formula that
\begin{equation*}
\big\|\nabla u_0\big\|^2_{L^2(\Omega\setminus\Sigma_{p_1\varepsilon})}
= \int_0^{p_1\varepsilon}\int_{S_t} |\nabla u_0(y)|^2 dS_t(y)dt
\leq C\int_0^{p_1\varepsilon}\frac{dt}{t^{2\sigma-2}}
\leq C\varepsilon^{2\sigma-1},
\end{equation*}
whenever $\sigma\in(1/2,1)$. And hence this directly leads to $\eqref{pri:3.11}$.

Now we turn to show the estimate $\eqref{pri:3.12}$.
Proceeding as in the proof of \cite[Lemma 3.5]{QX3},
we let $u_0 = v+w$ and $p_0=p_{0,1}+p_{0,2}$,
and they satisfy
\begin{equation*}
\text{(i)}\left\{\begin{aligned}
\mathcal{L}_0(v) + \nabla p_{0,1} & = \tilde{F} &\quad& \text{in}~~\mathbb{R}^d,\\
\text{div}(v) &= \tilde{h} &\quad& \text{in}~~\mathbb{R}^d,
\end{aligned}\right.
\qquad \text{and} \qquad
\text{(ii)}\left\{\begin{aligned}
\mathcal{L}_0(w) + \nabla p_{0,2} & = 0 &\quad& \text{in}~~\Omega,\\
\text{div}(w) &= 0 &\quad& \text{in}~~\Omega,\\
w &= g -v &\quad& \text{on}~\partial\Omega,
\end{aligned}\right.
\end{equation*}
where $\tilde{F}$ is the extension of $F$ to $\mathbb{R}^d$ by 0 outside of $\Omega$, and $\tilde{h}$ is the $W^{1,q}$-extension of $h$ to $\mathbb{R}^d$ such
that $\tilde{h} = h$ a.e. in $\Omega$ and $\|\tilde{h}\|_{W^{1,q}(\mathbb{R}^d)}\leq C\|h\|_{W^{1,q}(\Omega)}$.
Due to the singular integral estimates (see \cite[Lemma 3.4]{QX3}) for Stokes system with constant coefficients, we have
\begin{equation}\label{f:3.3}
\|\nabla^2 v\|_{L^{p}(\mathbb{R}^d)} +
\|\nabla v\|_{L^{p}(\mathbb{R}^d)}
+\|p_{0,1}\|_{W^{1,p}(\mathbb{R}^d)}
\leq C\Big\{\|F\|_{L^p(\Omega)}+\|h\|_{W^{1,p}(\Omega)}\Big\}.
\end{equation}
To deal with the system (ii), as in the proof $\eqref{f:3.2}$, we first obtain
\begin{equation}\label{f:3.4}
\begin{aligned}
|\nabla w(x)|
&\leq \frac{C}{r}\Big(\dashint_{B(x,r/2)}|w(y)-\dashint_{B(x,r/2)}w|^2dx\Big)^{1/2}
\leq Cr^{\sigma-1}[w]_{C^{0,\sigma}(\Omega)} \\
&\leq Cr^{\sigma-1}\Big\{\|g\|_{C^{0,1}(\partial\Omega)}
+\|v\|_{C^{0,1}(\partial\Omega)}\Big\}
\leq Cr^{\sigma-1}\Big\{\|g\|_{C^{0,1}(\partial\Omega)}
+\|\nabla v\|_{W^{1,p}(\mathbb{R}^d)}\Big\}
\leq Cr^{\sigma-1},
\end{aligned}
\end{equation}
where $p>d$, the fourth inequality above is a result of Sobolev imbedding theorem, and the last one follows from estimate $\eqref{f:3.3}$.
Also, it is clear to derive
the following interior estimate
\begin{equation*}
|\nabla^2 w(x)| \leq \frac{C}{r}
\Big(\dashint_{B(x,r/8)}|\nabla w|^2 dy\Big)^{\frac{1}{2}}.
\end{equation*}
Then we arrive at
\begin{equation*}
\begin{aligned}
\int_{\Sigma_{p_2\varepsilon}\setminus\Sigma_{c_0}}|\nabla^2 w|^2 dx
\leq C\int_{\Sigma_{p_2\varepsilon}\setminus\Sigma_{c_0}}\dashint_{B(x,\delta(x)/8)}
\frac{|\nabla w(y)|^2}{[\delta(x)]^2} dy dx
\leq C\int_{p_2\varepsilon}^\infty\frac{dt}{t^{4-2\sigma}}
\leq C\varepsilon^{2\sigma-3},
\end{aligned}
\end{equation*}
where we use the observation that $\delta(y)\approx \delta(x)$,
and the estimate $\eqref{f:3.4}$. This together with the estimates
$\|\nabla^2 v\|_{L^2(\Sigma_{p_2\varepsilon})}\leq C$ and
$\|\nabla^2 w\|_{L^2(\Sigma_{c_0})}\leq C$ gives
\begin{equation*}
\|\nabla^2u_0\|_{L^2(\Sigma_{p_2\varepsilon})}
\leq C\varepsilon^{\sigma-\frac{3}{2}}
\end{equation*}
and the desired estimate $\eqref{pri:3.12}$ follows, and we have completed the proof.
\end{proof}

\begin{flushleft}
\textbf{Proof of Theorem \ref{thm:3.3}.}
The following estimate can be derived by using the properties of smoothing operator and dual correctors as we stated in Section 2, details can be found in \cite[Lemma 3.2]{QX3}, for example.
\end{flushleft}
\begin{equation*}
\begin{aligned}
\|u_\varepsilon-u_0\|_{L^2(\Omega)} +
\|p_\varepsilon-p_0
& -\pi(\cdot/\varepsilon)S_\varepsilon(\psi_{2\varepsilon}
\nabla u_0)\|_{L^2(\Omega)/\mathbb{R}}\\
&\leq C\Big\{\|\nabla u_0\|_{L^2(\Omega\setminus4\varepsilon)}
+\varepsilon\|\nabla^2 u_0 \|_{L^{2}(\Sigma_{2\varepsilon})}\Big\},
\end{aligned}
\end{equation*}
Substituting estimates $\eqref{pri:3.11}$ and $\eqref{pri:3.12}$ into the right hand side of the above estimate leads to
the desired estimate $\eqref{pri:3.13}$, and the proof is complete.
\qed

\begin{corollary}
Assume the same conditions as in Theorem $\ref{thm:3.3}$. If we additionally assume that
the corrector $\pi$ is bounded, then we have
\begin{equation}\label{pri:3.14}
\|p_\varepsilon-p_0
 -\pi(\cdot/\varepsilon)\psi_{2\varepsilon}
\nabla u_0\|_{L^2(\Omega)/\mathbb{R}}
\leq C\varepsilon^{\sigma-\frac{1}{2}}\Big\{\|F\|_{L^p(\Omega)}
+ \|h\|_{W^{1,p}(\Omega)}
+\|g\|_{C^{0,1}(\partial\Omega)}\Big\},
\end{equation}
where $\sigma\in(1/2,1)$ is given in Theorem $\ref{thm:3.3}$, and $C$ depends on
$\mu,d,\sigma,p,\eta,\tau, M$ and $\|\pi\|_{L^\infty(Y)}$.
\end{corollary}

\begin{proof}
The result follows from Theorem $\ref{thm:3.3}$, Lemma $\ref{lemma:2.4}$ and Lemma $\ref{lemma:3.4}$.
\end{proof}

\begin{lemma}\label{lemma:3.5}
Let $\Omega$ be a bounded $C^{1,\tau}$ domain.
Given $F\in L^p(D_2;\mathbb{R}^d)$ and $h\in W^{1,p}(D_2)$ with $p>d$, and
$g\in C^{1,\eta}(\Delta_2)$ with $g(0) = 0$ and $\eta\in(0,\tau)$, let
$(u_0,p_0)\in H^{1}(\Omega;\mathbb{R}^d)\times L^{2}(\Omega)$ be the solution of
$\mathcal{L}_0(u_0) + \nabla p_0 = F$ and
$\emph{div}(u_0) = h$ in $D_{2}$ with
$u_0 = u_\varepsilon$ on $\partial D_2$, where $(u_\varepsilon,p_\varepsilon)$ satisfies
$\mathcal{L}_\varepsilon(u_\varepsilon) + \nabla p_\varepsilon = F$ and
$\emph{div}(u_\varepsilon)= h$ in $D_2$, and
$u_\varepsilon = g$ on $\Delta_2$.
Then for any $5\varepsilon\leq r< (1/4)$, there holds
\begin{equation}\label{pri:3.15}
\begin{aligned}
\bigg(\frac{1}{|D_r|}\int_{D_r\setminus\Sigma_{4\varepsilon}} |\nabla u_0|^2  dx\bigg)^{\frac{1}{2}}
&\leq C\left(\frac{\varepsilon}{r}\right)^{\frac{1}{2}}\bigg\{
\Big(\dashint_{D_{1}}|u_0|^2 dx\Big)^{\frac{1}{2}} + \Big(\dashint_{D_{1}}|F|^p dx\Big)^{\frac{1}{p}}\\
&+\Big(\dashint_{D_{1}}|\nabla h|^p dx\Big)^{\frac{1}{p}}
+\|h\|_{L^\infty(D_1)}
+\|\nabla g\|_{C^{0,\eta}(\Delta_{1})}\bigg\},
\end{aligned}
\end{equation}
and
\begin{equation}\label{pri:3.16}
\begin{aligned}
\bigg(\frac{1}{|D_r|}\int_{D_r\cap\Sigma_{4\varepsilon}}|\nabla^2 u_0|^2 dx\bigg)^{\frac{1}{2}}
&\leq C\left(\frac{1}{\varepsilon r}\right)^{\frac{1}{2}}\bigg\{
\Big(\dashint_{D_{2}}|u_0|^2 dx\Big)^{\frac{1}{2}}+ \Big(\dashint_{D_{2}}|F|^p dx\Big)^{\frac{1}{p}} \\
&+\Big(\dashint_{D_{2}}|\nabla h|^p dx\Big)^{\frac{1}{p}}
+\|h\|_{L^\infty(D_2)}
+\|\nabla g\|_{C^{0,\eta}(\Delta_{2})}\bigg\},
\end{aligned}
\end{equation}
where $C$ depends on $\mu,d,\eta,\tau,p$ and $M$.
\end{lemma}

\begin{proof}
The argument is quite similar to the one used in Lemma $\ref{lemma:3.4}$,
and we will simply sketch the proof here.
First, it is clear to see that the estimate $\eqref{pri:3.15}$ could be derived
from
\begin{equation*}
\|\nabla u_0\|_{L^\infty(D_{1/2})}
\leq C\bigg\{\Big(\dashint_{D_{1}}|u_0|^2 dx\Big)^{\frac{1}{2}}
+\Big(\dashint_{D_{1}}|F|^p dx\Big)^{\frac{1}{p}}
+\|h\|_{W^{1,p}(D_{1})}
+\|\nabla g\|_{C^{0,\eta}(\Delta_1)}\bigg\},
\end{equation*}
which follows from the estimate $\eqref{pri:2.17}$.
Once again, we let $u_0 = v+w$ and $p_0 = p_{0,1}+p_{0,2}$, which satisfy
\begin{equation*}
\text{(i)}\left\{\begin{aligned}
\mathcal{L}_0(v) + \nabla p_{0,1} & = \tilde{F} &~& \text{in}~\mathbb{R}^d,\\
\text{div}(v) &= \tilde{h} &~& \text{in}~\mathbb{R}^d,
\end{aligned}\right.
\quad \text{and} \quad
\text{(ii)}\left\{\begin{aligned}
\mathcal{L}_0(w) + \nabla p_{0,2} & = 0 &~& \text{in}~~D_2,\\
\text{div}(w) &= 0 &~& \text{in}~~D_2,\\
w &= u_\varepsilon - v &~& \text{on}~\partial D_2.
\end{aligned}\right.
\end{equation*}
Here $\tilde{F}$ and $\tilde{h}$ are the same extensions of $F$ and $h$ respectively as in the proof of Lemma $\ref{lemma:3.4}$. Also
it follows from the local boundary estimate $\eqref{pri:2.17}$ that
\begin{equation}
\begin{aligned}
\|\nabla w\|_{L^\infty(D_{1/2})}
&\leq C\Big\{\|w\|_{L^2(D_{1})} + \|\nabla g\|_{C^{0,\eta}(\Delta_1)}
+ \|\nabla v\|_{C^{0,\eta}(\Delta_1)}\Big\}\\
&\leq C\Big\{\|u_0\|_{L^2(D_{1})} + \|\nabla g\|_{C^{0,\eta}(\Delta_1)}
+ \|F\|_{L^p(D_2)}
+ \|h\|_{W^{1,p}(D_2)}\Big\} =:C_D,
\end{aligned}
\end{equation}
where we use the Sobolev imbedding theorem
and the estimate $\eqref{f:3.3}$ in the last step, as well as the fact that
$w=u_0-v$ in $D_1$ and $v-v(0)$ still satisfies (i). Then we have
\begin{equation*}
\begin{aligned}
\int_{\Sigma_{4\varepsilon}\cap D_r}|\nabla^2 w|^2 dx
\leq C\int_{\Sigma_{4\varepsilon}\cap D_r}\dashint_{B(x,\delta(x)/8)}
\frac{|\nabla w(y)|^2}{[\delta(x)]^2} dy dx
\leq CC_Dr^{d-1}\int_{4\varepsilon}^r\frac{dt}{t^{2}}
\leq \frac{CC_Dr^{d-1}}{\varepsilon},
\end{aligned}
\end{equation*}
and this implies the estimate
\begin{equation}\label{f:3.5}
\begin{aligned}
\bigg(\frac{1}{|D_r|}\int_{D_r\cap\Sigma_{4\varepsilon}}|\nabla^2 w|^2 dx\bigg)^{\frac{1}{2}}
&\leq C\left(\frac{1}{\varepsilon r}\right)^{\frac{1}{2}}\bigg\{
\Big(\dashint_{D_{1}}|u_0|^2 dx\Big)^{\frac{1}{2}}+ \Big(\dashint_{D_{2}}|F|^p dx\Big)^{\frac{1}{p}} \\
&+\Big(\dashint_{D_{2}}|\nabla h|^p dx\Big)^{\frac{1}{p}}
+\|h\|_{L^\infty(D_2)}
+\|\nabla g\|_{C^{0,\eta}(\Delta_{2})}\bigg\}.
\end{aligned}
\end{equation}
Moreover, using the estimate $\eqref{f:3.3}$ again, we derive
\begin{equation*}
\begin{aligned}
\bigg(\frac{1}{|D_r|}\int_{D_r\cap\Sigma_{4\varepsilon}}|\nabla^2 v|^2 dx\bigg)^{\frac{1}{2}}
\leq \Big(\dashint_{D_r}|\nabla^2 v|^p dx\Big)^{\frac{1}{p}}
\leq Cr^{-\frac{d}{p}}\bigg\{
\Big(\dashint_{D_{2}}|F|^p dx\Big)^{\frac{1}{p}}
+\|h\|_{W^{1,p}(D_2)}\bigg\}.
\end{aligned}
\end{equation*}
By noting that $p>d$ and $\varepsilon<r<1$, this together with $\eqref{f:3.5}$ leads to the desired estimate
$\eqref{pri:3.16}$, and we have completed the proof.
\end{proof}

\section{Lipschitz estimates on velocity term}
We will use this section to provide the boundary Lipschitz estimate of the velocity, which is stated in the following theorem.
\begin{thm}\label{thm:5.2}
Let $\Omega$ be a bounded $C^{1,\tau}$ domain.
Suppose $A$ satisfies $\eqref{a:1}-\eqref{a:3}$. Let $(u_\varepsilon, p_\varepsilon)\in H^1(D_5;\mathbb{R}^d)
\times L^2(D_5)$ be a weak solution of
$\mathcal{L}_\varepsilon(u_\varepsilon) + \nabla p_\varepsilon = F$
and $\emph{div}(u_\varepsilon) = h$ in
$D_5$ and $u_\varepsilon = g$ on $\Delta_5$,
where $F\in L^p(D_5;\mathbb{R}^d)$ and $h\in W^{1,p}(D_5)$ with
$p>d$, and $g\in C^{1,\eta}(\Delta_5;\mathbb{R}^d)$ with $g(0) = 0$,
and $0<\eta<\min\{\tau,1-d/p\}$.
Then there holds
\begin{equation}\label{pri:5.9}
\begin{aligned}
\Big(\dashint_{D_r} |\nabla u_\varepsilon|^2dx\Big)^{\frac{1}{2}}
&\leq C\bigg\{\Big(\dashint_{D_1} |u_\varepsilon|^2dx\Big)^{\frac{1}{2}}+ \Big(\dashint_{D_1} |F|^pdx\Big)^{\frac{1}{p}} \\
&+ \Big(\dashint_{D_1} |\nabla h|^pdx\Big)^{\frac{1}{p}}
+ \|h\|_{L^\infty(D_1)}
+ \|\nabla g\|_{C^{0,\eta}(\Delta_1)}\bigg\}
\end{aligned}
\end{equation}
for any $0<r<(1/4)$,
where $C$ depends only on $\mu,\omega,d,M,\sigma$ and $p$.
\end{thm}
Before we give the proof, we recall the following uniform boundary H\"older estimate of the velocity, obtained by compactness argument in \cite[Theorem 6.1]{SGZWS}.
\begin{thm}[Boundary H\"older estimates]
Let $\Omega$ be a bounded $C^1$ domain. Suppose $A$ satisfies $\eqref{a:1}$ and $\eqref{a:2}$. Let
$(u_\varepsilon,p_\varepsilon)\in H^1(D_5;\mathbb{R}^d)\times L^2(D_5)$ be the weak solution to
$\mathcal{L}_\varepsilon(u_\varepsilon) + \nabla p_\varepsilon = F$
and $\emph{div}(u_\varepsilon)=h$ in $D_5$, and $u_\varepsilon =g$ on $\Delta_4$, where
$F\in L^p(D_5;\mathbb{R}^d)$, $h\in W^{1,p}(D_5)$ with $p>d$, and $g\in C^{0,1}(\Delta_5;\mathbb{R}^d)$
with $g(0) = 0$. Then for any $0<\sigma<1$, there holds
\begin{equation}\label{pri:4.6}
\Big(\dashint_{D_r}|u_\varepsilon|^2dx\Big)^{\frac{1}{2}}
\leq Cr^\sigma\bigg\{
\Big(\dashint_{D_1}|u_\varepsilon|^2dx\Big)^{\frac{1}{2}}
+\|h\|_{W^{1,p}(D_1)} + \|\nabla g\|_{C^{0,\eta}(\Delta_1)}\bigg\}
\end{equation}
for any $\varepsilon\leq r<(1/4)$, where $C$ depends only on $\mu,p,d,\sigma,M$.
\end{thm}

\begin{lemma}\label{lemma:5.4}
Let $\varepsilon\leq r<1$.
Assume the same conditions as in Theorem $\ref{thm:5.2}$.
Let $(u_\varepsilon, p_\varepsilon)\in H^1(D_5;\mathbb{R}^d)
\times L^2(D_5)$ be a weak solution of
$\mathcal{L}_\varepsilon(u_\varepsilon) + \nabla p_\varepsilon = F$ and
$\emph{div}(u_\varepsilon) = h$ in
$D_5$, and $u_\varepsilon = g$ on $\Delta_5$.  Then there exists
$(u_0,p_0)\in H^1(D_2;\mathbb{R}^d)\times L^2(D_2)$ such that
$\mathcal{L}_0(u_0) + \nabla p_0 = F$ and $\emph{div}(u_0) = h$ in $D_2$ and
$u_0 = g$ on $\Delta_2$, and for some $\lambda>0$, then
\begin{equation}\label{pri:5.11}
\begin{aligned}
 \Big(\dashint_{D_r} |u_\varepsilon - u_0|^2  dx\Big)^{\frac{1}{2}}
 &\leq C\left(\frac{\varepsilon}{r}\right)^{\lambda}
 \bigg\{\Big(\dashint_{D_{2r}}|u_\varepsilon|^2 dx\Big)^{\frac{1}{2}}
 + r^2\Big(\dashint_{D_{2r}} |F|^pdx\Big)^{\frac{1}{p}} + r\|h\|_{L^{\infty}(D_{2r})}\\
 &\quad+ r^2\Big(\dashint_{D_{2r}}|\nabla h|^p dx\Big)^{\frac{1}{p}}
 + r\|\nabla g\|_{L^{\infty}(\Delta_{2r})}
 + r^{1+\eta}\big[\nabla g\big]_{C^{0,\eta}(\Delta_{2r})}\bigg\},
\end{aligned}
\end{equation}
where $\lambda = 2\sigma-\frac{3}{2}$,
and $C$ depends only on $\mu,\omega,\lambda,\tau,\eta,M,\sigma$ and $d$.
\end{lemma}

\begin{proof}
By rescaling argument, we may assume $r=1$. Since $\text{div}(u_\varepsilon) = h$ in $D_2$,
there exists $(u_0,p_0)\in H^1(D_2;\mathbb{R}^d)\times L^2(D_2)$ satisfying
$\mathcal{L}_0(u_0)+\nabla p_0 = F$, and $\text{div}(u_0) = h$ in $D_2$,
and $u_0 = u_\varepsilon$ on $\partial D_2$.
In view of Theorem $\ref{thm:3.3}$, we have
\begin{equation}\label{f:5.8}
\begin{aligned}
\|u_\varepsilon - u_0\|_{L^2(D_1)}
\leq C\varepsilon^{\sigma-\frac{1}{2}}\bigg\{
 \|F\|_{L^{p}(D_2)} + \|h\|_{W^{1,p}(D_2)}
+\|\nabla g\|_{L^{\infty}(\Delta_2)}
+\|\nabla u_\varepsilon\|_{L^\infty(\partial D_2\setminus\Delta_2)}\bigg\},
\end{aligned}
\end{equation}
and it remains to estimate the last term in the right-hand side of $\eqref{f:5.8}$.
It is clear to see that $\partial D_2\setminus\Delta_2$ may be covered by
$\{\tilde{D}_{4\varepsilon}\}$ and $\{\tilde{B}_\varepsilon\}$.
Hence it follows from the local estimate $\eqref{pri:2.17}$ that
\begin{equation}\label{f:5.12}
\begin{aligned}
\|\nabla u_\varepsilon\|_{L^\infty(\tilde{D}_{4\varepsilon})}
&\leq C\bigg\{\frac{1}{\varepsilon}\Big(\dashint_{\tilde{D}_{8\varepsilon}}|u_\varepsilon|^2dx\Big)^{\frac{1}{2}}
  +\varepsilon\Big(\dashint_{\tilde{D}_{8\varepsilon}}|F|^pdx\Big)^{\frac{1}{p}}
  +\varepsilon\Big(\dashint_{\tilde{D}_{8\varepsilon}}|\nabla h|^pdx\Big)^{\frac{1}{p}} \\
& \qquad\qquad\qquad\qquad +\|h\|_{L^\infty(D_{4})}
  +\|\nabla g\|_{L^\infty(\Delta_{4})}
  +\varepsilon^{\eta}[\nabla g]_{C^{0,\eta}(\Delta_{4})}\bigg\} \\
&\leq C\varepsilon^{\sigma-1}\bigg\{\Big(\dashint_{D_{4}}|u_\varepsilon|^2 dx\Big)^{\frac{1}{2}}
+\Big(\dashint_{D_{4}}|F|^p dx\Big)^{\frac{1}{p}}
+ \|h\|_{W^{1,p}(D_{4})}
+ \|\nabla g\|_{C^{0,\eta}(\Delta_{4})}\bigg\},
\end{aligned}
\end{equation}
where the second inequality follows from the uniform H\"older estimate $\eqref{pri:4.6}$.
Then for any $\tilde{B}_\varepsilon(x)$, where $x\in\partial D_2\setminus\Delta_2$, there are
two cases: (1) $r=\text{dist}(x,\Delta_2)\in [\varepsilon,1/8]$ and (2)
$r> \frac{1}{8}$. Obviously, the second case follows from the
uniform interior Lipschitz estimates \cite[Corollary 1.2]{SGZWS} that
\begin{equation}\label{f:5.13}
\begin{aligned}
\|\nabla u_\varepsilon\|_{L^\infty(\tilde{B}_\varepsilon)}
&\leq C\bigg\{\Big(\dashint_{\tilde{B}_{\frac{1}{4}}}|\nabla u_\varepsilon|^2 dx\Big)^{\frac{1}{2}}
+ \|h\|_{C^{0,\rho}(\tilde{B}_{\frac{1}{4}})} + \|h\|_{L^\infty(\tilde{B}_{\frac{1}{4}})}
\bigg\} \\
&\leq C\bigg\{\Big(\dashint_{D_{4}}|u_\varepsilon|^2 dx\Big)^{\frac{1}{2}}
+\Big(\dashint_{D_{4}}|F|^p dx\Big)^{\frac{1}{p}}
+ \|h\|_{W^{1,p}(D_{4})}
+ \|\nabla g\|_{L^{\infty}(\Delta_{4})}\bigg\},
\end{aligned}
\end{equation}
where $\rho= 1-d/p$, we use Caccippoli's inequality \cite[Theorem 6.2]{SGZWS} and the Sobolev imbedding theorem
in the last inequality. We now turn to study the case (1):
\begin{equation}\label{f:5.14}
\begin{aligned}
\|\nabla u_\varepsilon\|_{L^\infty(\tilde{B}_{\varepsilon})}
&\leq \|\nabla u_\varepsilon\|_{L^\infty(\tilde{B}_{r/2})} \\
&\leq C\bigg\{\frac{1}{r}\Big(\dashint_{\tilde{D}_{2r}}|u_\varepsilon|^2dx\Big)^{\frac{1}{2}}
  +r\Big(\dashint_{\tilde{D}_{2r}}|F|^pdx\Big)^{\frac{1}{p}}
  +r\Big(\dashint_{\tilde{D}_{2r}}|\nabla h|^pdx\Big)^{\frac{1}{p}} \\
& \qquad\qquad\qquad\qquad +\|h\|_{L^\infty(D_{4})}
  +\|\nabla g\|_{L^\infty(\Delta_{4})}
  +r^{\eta}[\nabla g]_{C^{0,\eta}(\Delta_{4})}\bigg\} \\
&\leq C\varepsilon^{\sigma-1}\bigg\{\Big(\dashint_{D_{4}}|u_\varepsilon|^2 dx\Big)^{\frac{1}{2}}
+\Big(\dashint_{D_{4}}|F|^p dx\Big)^{\frac{1}{p}}
+ \|h\|_{W^{1,p}(D_{4})}
+ \|\nabla g\|_{C^{0,\eta}(\Delta_{4})}\bigg\},
\end{aligned}
\end{equation}
where we use the uniform H\"older estimate $\eqref{pri:4.6}$, as well as the fact that $r>\varepsilon$,
in the last inequality.

Consequently, combining $\eqref{f:5.8}$, $\eqref{f:5.12}$, $\eqref{f:5.13}$ and $\eqref{f:5.14}$, we have
\begin{equation*}
\begin{aligned}
\big\|u_\varepsilon - u_0\big\|_{L^2(D_1)}
&\leq C\varepsilon^{\lambda}
\bigg\{\Big(\dashint_{D_{4}}|\nabla u_\varepsilon|^2 dx\Big)^{\frac{1}{2}}
+ \Big(\dashint_{D_{4}}|F|^p dx\Big)^{\frac{1}{p}}
+ \|h\|_{W^{1,p}(D_4)}
+ \|\nabla g\|_{C^{0,\eta}(\Delta_4)} \bigg\},
\end{aligned}
\end{equation*}
where $\lambda = 2\sigma-\frac{3}{2}$.
By rescaling argument we can derive
the desired estimate $\eqref{pri:5.11}$, and we have completed the proof.
\end{proof}

Before we proceed further, for any matrix $M\in \mathbb{R}^d$, we denote $G(r,v)$ as the following
\begin{equation}
\begin{aligned}
G(r,v) &= \frac{1}{r}\inf_{M\in\mathbb{R}^{d\times d}}
\Bigg\{\Big(\dashint_{D_r}|v-Mx|^2dx\Big)^{\frac{1}{2}}
+ r^2\Big(\dashint_{D_r}|F|^pdx\Big)^{\frac{1}{p}}+r\|h-\text{Tr}(M)\|_{L^\infty(D_r)} \\
&\qquad\quad +r^2\Big(\dashint_{D_r}|\nabla h|^pdx\Big)^{\frac{1}{p}}+ r\big\|\nabla_T (g-Mx)\big\|_{L^\infty(\Delta_{r})}
+r^{1+\eta}\big[\nabla_T(g-Mx)\big]_{C^{0,\eta}(\Delta_r)}\Bigg\},
\end{aligned}
\end{equation}
where $\text{Tr}(M)$ denotes the trace of $M$.

\begin{lemma}\label{lemma:5.5}
Let $(u_0,p_0)\in H^1(D_4;\mathbb{R}^d)\times L^2(D_4)$ be a solution of
$\mathcal{L}_0(u_0)+\nabla p_0 = F$ and $\emph{div}(u_0) = h$ in $D_4$, and
$u_0 = g$, where $g\in C^{0,1}(\Delta_4)$ with $g(0)=0$. Then
there exists $\theta\in(0,1/4)$, depending on $\mu,d,\eta$ and $M$, such that
\begin{equation}\label{pri:5.10}
 G(\theta r,u_0)\leq \frac{1}{2} G(r,u_0)
\end{equation}
holds for any $r\in(0,1)$.
\end{lemma}

\begin{proof}
We may assume $r=1$ by rescaling argument. By the definition of $G(\theta,u_0)$, we see that
\begin{equation*}
\begin{aligned}
G(\theta,u_0) &\leq  \frac{1}{\theta}
\bigg\{\Big(\dashint_{D_\theta}|u_0-M_0x|^2dx\Big)^{\frac{1}{2}}
+ \theta^2\Big(\dashint_{D_\theta}|F|^pdx\Big)^{\frac{1}{p}}+\theta\|h-\text{Tr}(M_0)\|_{L^\infty(D_\theta)} \\
&\quad+\theta^2\Big(\dashint_{D_\theta}|\nabla h|^pdx\Big)^{\frac{1}{p}}
+ \theta\big\|\nabla_T (g-M_0x)\big\|_{L^\infty(\Delta_{\theta})}
+\theta^{1+\eta}\big[\nabla_T(g-M_0x)\big]_{C^{0,\eta}(\Delta_\theta)}\bigg\} \\
&\leq \theta^{\eta}\bigg\{[\nabla u_0]_{C^{0,\eta}(D_{1/2})}
+\Big(\dashint_{D_{1/2}}|F|^pdx\Big)^{\frac{1}{p}}
+\Big(\dashint_{D_{1/2}}|\nabla h|^pdx\Big)^{\frac{1}{p}}
\bigg\},
\end{aligned}
\end{equation*}
where we choose $M_0 = \nabla u_0(0)$. For any
$M\in \mathbb{R}^{d\times d}$, we let $\tilde{u}_0 = u_0 - Mx$. Clearly it satisfies the
system: $\mathcal{L}_0(\tilde{u}_0)+\nabla p_0 = F$, and $\text{div}(\tilde{u}_0)
=h-\text{Tr}(M)$ in $D_4$, $\tilde{u}_0 = g-Mx$ on
$\Delta_4$. Hence it follows from boundary Schauder estimates $\eqref{pri:2.18}$ that
\begin{equation*}
\begin{aligned}
\big[\nabla u_0\big]_{C^{0,\eta}(D_{1/2})}
= \big[\nabla \tilde{u}_0\big]_{C^{0,\eta}(D_{1/2})}
\leq CG(1,u_0).
\end{aligned}
\end{equation*}
It is clear to see that there exists $\theta\in(0,1/4)$ such that
$G(\theta,u_0)\leq \frac{1}{2}G(1,u_0)$.
Then the desire result $\eqref{pri:5.10}$ can be obtained simply by a rescaling argument.
\end{proof}

For simplicity, we also denote $\Phi(r)$ by
\begin{equation*}
\begin{aligned}
\Phi(r) = \frac{1}{r}\bigg\{
\Big(\dashint_{D_{r}}|u_\varepsilon|^2 dx\Big)^{\frac{1}{2}}
 &+ r^2\Big(\dashint_{D_{r}} |F|^pdx\Big)^{\frac{1}{p}}+ r\|h\|_{L^{\infty}(D_{r})}
  \\
&+ r^2\Big(\dashint_{D_{r}}|\nabla h|^p dx\Big)^{\frac{1}{p}}+ r\|\nabla g\|_{L^{\infty}(\Delta_{2r})}
 + r^{1+\eta}\big[\nabla g\big]_{C^{0,\eta}(\Delta_{r})}\bigg\}.
\end{aligned}
\end{equation*}

\begin{lemma}\label{lemma:5.6}
Let $\lambda$ be given in Lemma $\ref{lemma:5.4}$. Assume the same conditions as in Theorem $\ref{thm:5.2}$.
Let $(u_\varepsilon,p_\varepsilon)$ be the solution of
$\mathcal{L}_\varepsilon(u_\varepsilon)+\nabla p_\varepsilon = F$ and
$\emph{div}(u_\varepsilon) = h$ in $D_5$ with $u_\varepsilon= g$ on
$\Delta_5$. Then we have
\begin{equation}
 G(\theta r, u_\varepsilon) \leq \frac{1}{2}G(r,u_\varepsilon)
 + C\left(\frac{\varepsilon}{r}\right)^\lambda\Phi(2r)
\end{equation}
for any $r\in[\varepsilon,1/2]$, where $\theta\in(0,1/4)$ is given in Lemma $\ref{lemma:5.5}$.
\end{lemma}

\begin{proof}
Fix $r\in[\varepsilon,1/2]$, let $(u_0,p_0)$ be a solution to
$\mathcal{L}_0(u_0)+\nabla p_0 = F$  and $\text{div}(u_0) = h$ in $D_r$,
and $u_0 = u_\varepsilon$ on
$\partial D_r$. Then we have
\begin{equation*}
\begin{aligned}
G(\theta r,u_\varepsilon)
&\leq \frac{1}{\theta r}\Big(\dashint_{D_{\theta r}}|u_\varepsilon - u_0|^2 dx\Big)^{\frac{1}{2}}
+ G(\theta r, u_0) \\
&\leq \frac{C}{r}\Big(\dashint_{D_{r}}|u_\varepsilon - u_0|^2 dx\Big)^{\frac{1}{2}}
+ \frac{1}{2}G(r, u_0)\\
&\leq \frac{1}{2}G(r, u_\varepsilon)
+ \frac{C}{r}\Big(\dashint_{D_{r}}|u_\varepsilon - u_0|^2 dx\Big)^{\frac{1}{2}}\\
&\leq  \frac{1}{2}G(r, u_\varepsilon) + C(\varepsilon/r)^\lambda\Phi(2r),
\end{aligned}
\end{equation*}
where we use the estimate $\eqref{pri:5.10}$ in the second inequality,
and $\eqref{pri:5.11}$ in the last one. The proof is complete.
\end{proof}

\begin{lemma}\label{lemma:5.3}
Let $\Psi(r)$ and $\psi(r)$ be two nonnegative continuous functions on the integral $(0,1]$.
Let $0<\varepsilon<\frac{1}{4}$. Suppose that there exists a constant $C_0$ such that
\begin{equation}\label{pri:5.4}
\left\{\begin{aligned}
  &\max_{r\leq t\leq 2r} \Psi(r) \leq C_0 \Psi(2r),\\
  &\max_{r\leq s,t\leq 2r} |\psi(r)-\psi(s)|\leq C_0 \Psi(2r)
  \end{aligned}\right.
\end{equation}
for any $r\in[\varepsilon,1/2]$. We further assume that
\begin{equation}\label{pri:5.5}
\Psi(\theta r)\leq \frac{1}{2}\Psi(r) + C_0w(\varepsilon/r)\Big\{\Psi(2r)+\psi(2r)\Big\}
\end{equation}
holds for any $r\in[\varepsilon,1/2]$, where $\theta\in(0,1/4)$ and $\omega$ is a nonnegative
increasing function in $[0,1]$ such that $\omega(0)=0$ and
\begin{equation}\label{pri:5.6}
 \int_0^1 \frac{w(t)}{t} dt <\infty.
\end{equation}
Then, we have
\begin{equation}\label{pri:5.7}
\max_{\varepsilon\leq r\leq 1}\Big\{\Psi(r)+\psi(r)\Big\}
\leq C\Big\{\Psi(1)+\psi(1)\Big\},
\end{equation}
where $C$ depends only on $C_0, \theta$ and $w$.
\end{lemma}

\begin{proof}
See \cite[Lemma 8.5]{SZW12}.
\end{proof}

\begin{flushleft}
\textbf{Proof of Theorem $\ref{thm:5.2}$}.
It is fine to assume $0<\varepsilon<1/4$, otherwise it follows from the classical theory.
In view of Lemma $\ref{lemma:5.3}$,
we set $\Psi(r) = G(r,u_\varepsilon)$, $w(t) =t^\lambda$,
where $\lambda>0$ is given in Lemma $\ref{lemma:5.4}$. In order to prove the desired estimate
$\eqref{pri:5.9}$, it is sufficient to verify $\eqref{pri:5.4}$ and $\eqref{pri:5.5}$. Let $\psi(r) = |M_r|$, where $M_r$ is the matrix associated with $\Psi(r)$, i.e., in the following sense,
\begin{equation*}
\begin{aligned}
\Psi(r) &= \frac{1}{r}
\Bigg\{\Big(\dashint_{D_r}|u_\varepsilon-M_r x|^2dx\Big)^{\frac{1}{2}}
+ r^2\Big(\dashint_{D_r}|F|^pdx\Big)^{\frac{1}{p}}+r\|h-\text{Tr}(M_r)\|_{L^\infty(D_r)} \\
&\qquad\quad +r^2\Big(\dashint_{D_r}|\nabla h|^pdx\Big)^{\frac{1}{p}}+ r\big\|\nabla_T (g-M_r x)\big\|_{L^\infty(\Delta_{r})}
+r^{1+\eta}\big[\nabla_T(g-M_r x)\big]_{C^{0,\eta}(\Delta_r)}\Bigg\},
\end{aligned}
\end{equation*}
Then we have,
\begin{equation*}
 \Phi(2r) \leq \Psi(2r) + \psi(2r).
\end{equation*}
\end{flushleft}
This together with Lemma $\ref{lemma:5.6}$ gives
\begin{equation*}
\Psi(\theta r)\leq \frac{1}{2}\Psi(r) + C_0 w(\varepsilon/r)\Big\{\Psi(2r)+\psi(2r)\Big\},
\end{equation*}
which satisfies the condition $\eqref{pri:5.5}$ in Lemma $\ref{lemma:5.3}$.
Let $t,s\in [r,2r]$, and $v(x)=(M_t-M_s)x$. It is clear to see $v$ is harmonic in $\mathbb{R}^d$.
Since $D_r$ satisfies the interior ball condition, we arrive at
\begin{equation}\label{f:5.15}
\begin{aligned}
|M_t-M_s|&\leq \frac{C}{r}\Big(\dashint_{D_r}|(M_t-M_s)x|^2dx\Big)^{\frac{1}{2}}\\
&\leq \frac{C}{t}\Big(\dashint_{D_t}|u_\varepsilon - M_tx|^2dx\Big)^{\frac{1}{2}}
+ \frac{C}{s}\Big(\dashint_{D_s}|u_\varepsilon - M_sx|^2dx\Big)^{\frac{1}{2}}\\
&\leq C\Big\{\Psi(t)+\Psi(s)\Big\}\leq C\Psi(2r),
\end{aligned}
\end{equation}
where the second and the last steps are based on the fact that $s,t\in[r,2r]$. Due to the same reason, it
is easy to obtain $\Psi(r)\leq C\Psi(2r)$, and the estimate $\eqref{f:5.15}$ admits the condition
$\eqref{pri:5.4}$. Besides, $w$ here obviously satisfies the condition $\eqref{pri:5.6}$.
Hence, according to Lemma $\ref{lemma:5.3}$, for any $r\in[\varepsilon,1/4]$,
we have the following estimate
\begin{equation}\label{f:5.16}
\frac{1}{r}\Big(\dashint_{D_{2r}}|u_\varepsilon|^2 dx\Big)^{\frac{1}{2}}
\leq C\Big\{\Psi(2r) + \psi(2r) \Big\}
\leq C\Big\{\Psi(1) + \psi(1) \Big\}.
\end{equation}
Hence, for $\varepsilon\leq r<(1/4)$, the desired estimate $\eqref{pri:5.9}$ consequently follows from
$\eqref{f:5.16}$ and the Cacciopoli's inequality \cite[Theorem 6.2]{SGZWS}.
Obviously, the case of $0<r<\varepsilon$ can be done simply by the blow-up argument, and
we have completed the proof.
\qed

\section{$L^\infty$ estimates on pressure term}

Now we move on to provide the boundary $L^\infty$ estimates on the pressure term.
\begin{thm}\label{thm:6.2}
Assume the same conditions as in Theorem $\ref{thm:1.1}$.
Let $(u_\varepsilon,p_\varepsilon)\in H^1(D_5;\mathbb{R}^d)\times L^2(D_5)$
be the solution of $\mathcal{L}_\varepsilon(u_\varepsilon)+\nabla p_\varepsilon = F$
and $\emph{div}(u_\varepsilon) = h$ in $D_5$ with
$u_\varepsilon = g$ on $\Delta_5$, where $F\in L^p(D_5;\mathbb{R}^d)$ and
$h\in W^{1,p}(D_5)$ with $p>d$, and $g\in C^{1,\eta}$ with $g(0)=0$ and $\eta\in(0,\tau]$.
Then there holds
\begin{equation}\label{pri:6.4}
\begin{aligned}
\Big|\dashint_{D_r}p_\varepsilon - \dashint_{D_1}p_\varepsilon dx\Big|
&\leq C\bigg\{\Big(\dashint_{D_2} |u_\varepsilon|^2dx\Big)^{\frac{1}{2}}  \\
&+ \Big(\dashint_{D_2} |F|^p  dx\Big)^{\frac{1}{p}}+ \Big(\dashint_{D_2} |\nabla h|^pdx\Big)^{\frac{1}{p}}+ \|h\|_{L^\infty(D_2)}
+ \|\nabla g\|_{C^{0,\eta}(\Delta_2)}\bigg\}
\end{aligned}
\end{equation}
for any $0< r<(1/4)$,
where the constant $C$ depends on $\mu,d,\eta,M$ and $p$.
\end{thm}

\begin{lemma}\label{lemma:6.2}
Let $(u_\varepsilon,p_\varepsilon)$ be the solution of
$\mathcal{L}_\varepsilon(u_\varepsilon)+\nabla p_\varepsilon = F$, $\emph{div}(u_\varepsilon) = h$
in $D_{5}$, and $u_\varepsilon = g$ on $\Delta_{5}$. Then there exists
$(u_0,p_0)\in H^1(D_{2};\mathbb{R}^d)\times L^2(D_{2})$ such that
$\mathcal{L}_0(u_0)+\nabla p_0 = F$ and $\emph{div}(u_0) = h$ in $D_{2}$ with $u_0=g$ on $\Delta_{2}$, and
there holds
\begin{equation}\label{pri:6.6}
\begin{aligned}
\Big(\dashint_{D_r}\big|& p_\varepsilon - p_0
-\pi(\cdot/\varepsilon)\psi_{4\varepsilon}\nabla u_0-c\big|^2 dx\Big)^{\frac{1}{2}}
\leq C\left(\frac{\varepsilon}{r}\right)^{\lambda}
\Bigg\{\frac{1}{r}\Big(\dashint_{D_{4r}}|u_\varepsilon|^2 dx\Big)^{\frac{1}{2}} \\
 + \|h\|_{L^{\infty}(D_{4r})}
&+ r\Big(\dashint_{D_{4r}}|\nabla h|^p dx\Big)^{\frac{1}{p}}
 + r\Big(\dashint_{D_{4r}} |F|^pdx\Big)^{\frac{1}{p}}
 + \|\nabla g\|_{L^{\infty}(\Delta_{4r})}
 + r^{\eta}\big[\nabla g\big]_{C^{0,\eta}(\Delta_{4r})}\Bigg\},
\end{aligned}
\end{equation}
for any $\varepsilon\leq r<(1/4)$, where $\lambda = 2\sigma-3/2$, and
\begin{equation*}
c = \dashint_{D_r} \Big[ p_\varepsilon - p_0
- \pi(\cdot/\varepsilon)\psi_{4\varepsilon}\nabla u_0 \Big] dx,
\end{equation*}
and $C$ depends only on $\mu,d,\eta$ and $M$.
\end{lemma}

\begin{proof}
This lemma as the counterpart of Lemma $\ref{lemma:5.4}$ obeys a similar proof.
By rescaling argument, we may prove it for $r=1$. Let
$(u_0,p_0)$ be the same one given in Lemma $\ref{lemma:5.4}$. Hence, it follows from the estimate $\eqref{pri:3.14}$ that
\begin{equation*}
\begin{aligned}
\|p_\varepsilon-p_0
-\pi(\cdot/\varepsilon)\psi_{2\varepsilon}
\nabla u_0&\|_{L^2(D_1)/\mathbb{R}}
\leq  \|p_\varepsilon-p_0 -\pi(\cdot/\varepsilon)\psi_{2\varepsilon}
\nabla u_0\|_{L^2(D_2)/\mathbb{R}}\\
&\leq C\varepsilon^{\sigma-\frac{1}{2}}\bigg\{
 \|F\|_{L^{p}(D_2)} + \|h\|_{W^{1,p}(D_2)}
+\|\nabla g\|_{L^{\infty}(\Delta_2)}
+\|\nabla u_\varepsilon\|_{L^\infty(\partial D_2\setminus\Delta_2)}\bigg\}.
\end{aligned}
\end{equation*}
This together with the estimates $\eqref{f:5.12}$, $\eqref{f:5.13}$ and $\eqref{f:5.14}$ gives
\begin{equation}\label{f:6.8}
\begin{aligned}
\|p_\varepsilon-p_0
-\pi(\cdot/\varepsilon)\psi_{2\varepsilon}
\nabla u_0\|_{L^2(D_1)/\mathbb{R}}&\leq C\varepsilon^{\lambda}\bigg\{
\|u_\varepsilon\|_{L^2(D_4)}
+ \|F\|_{L^{p}(D_4)} + \|h\|_{W^{1,p}(D_4)}
+\|\nabla g\|_{C^{0,\eta}(\Delta_4)}\bigg\},
\end{aligned}
\end{equation}
where $\lambda = 2\sigma-3/2$. The desired estimate $\eqref{pri:6.6}$ is derived from the rescaling argument for any
$r\in[\varepsilon,1/4]$.
\end{proof}

\begin{lemma}\label{lemma:6.3}
Assume the same conditions as in Theorem $\ref{thm:1.1}$.
Let $(u_0,p_0)\in H^1(D_4;\mathbb{R}^d)\times L^2(D_4)$ be the weak solution of
$\mathcal{L}_0(u_0) + \nabla p_0= F$ and $\emph{div}(u_0) = h$ in $D_4$ with
$u_0 = g$ on $\Delta_4$. Then for any $0<\rho<\min\{\eta,1-d/p\}$, there holds
\begin{equation}\label{pri:6.5}
 \Big(\dashint_{D_r} \big|p_0 - \dashint_{D_s} p_0\big|^2 dx\Big)^{\frac{1}{2}}
 \leq Cs^\rho\bigg\{\|u_0\|_{L^2(D_2)}
 +\|F\|_{L^{p}(D_2)}
 + \|h\|_{W^{1,p}(D_2)}
 +\|\nabla g\|_{C^{0,\eta}(\Delta_2)}\bigg\}
\end{equation}
for any $0<r\leq s\leq 1$, where $C$ depends only on $\mu,d,p,\eta,M$ and $\rho$.
\end{lemma}

\begin{proof}
Since $\mathcal{L}_0$ is an operator with constant coefficients, it is well known that
$(u_0,p_0)\in C^{1,\rho}(D_1;\mathbb{R}^d)\times C^{0,\rho}(D_1)$, and
it follows from the local estimate $\eqref{pri:2.18}$ that
\begin{equation}
\begin{aligned}
\dashint_{D_r} \big|p_0 - \dashint_{D_s} p_0\big|^2 dx
& = \dashint_{D_r} \Big|p_0 -p_0(0) - \dashint_{D_s}\big[p_0-p_0(0)\big]dy \Big|^2 dx \\
& \leq C\Big\{r^{2\rho}+s^{2\rho}\Big\}\big[p_0\big]_{C^{0,\rho}(D_1)}^2\\
& \leq Cs^{2\rho}\bigg\{\|u_0\|_{L^2(D_2)}
 +\|F\|_{L^{p}(D_2)}
 + \|h\|_{W^{1,p}(D_2)}
 +\|\nabla g\|_{C^{0,\eta}(\Delta_2)}\bigg\}.
\end{aligned}
\end{equation}
This implies the desired estimate $\eqref{pri:6.5}$. We have completed
the proof.
\end{proof}

\begin{flushleft}
\textbf{Proof of Theorem $\ref{thm:6.2}$.} For any $\varepsilon\leq r<(1/4)$,
there exists an integer $k>0$ such that $2^{-k-1}\leq r<2^{-k}$. Then we have
\end{flushleft}
\begin{equation*}
 \Big|\dashint_{D_r} p_\varepsilon - \dashint_{D_1} p_\varepsilon dx\Big|
 \leq 2\Big|\dashint_{D_{2^{-k}}} p_\varepsilon - \dashint_{D_1} p_\varepsilon dx \Big|
 \leq 4\sum_{j=1}^k \Big|\dashint_{D_{2^{-j}}} p_\varepsilon - \dashint_{D_{2^{-j+1}}} p_\varepsilon dx\Big|
\end{equation*}
It now remains to estimate each terms above. For simplicity, we denote $Z_\varepsilon$ and its average by
\begin{equation*}
Z_\varepsilon(x) = p_\varepsilon(x) - p_0(x)
 - \pi_i^\gamma(x/\varepsilon)\psi_{4\varepsilon}(x)\nabla_i u_0^\gamma(x),
\qquad
\overline{(Z_\varepsilon)}_{D_{2^{-j}}} = \dashint_{D_{2^{-j}}}Z_\varepsilon(x) dx.
\end{equation*}
Thus for any $c\in \mathbb{R}$, it follows that
\begin{equation}\label{f:6.9}
\begin{aligned}
\Big|\dashint_{D_{2^{-j}}} p_\varepsilon
- \dashint_{D_{2^{-j+1}}} p_\varepsilon \, dx\Big|
&\leq \Bigg|\dashint_{D_{2^{-j}}} \bigg\{\Big[Z_\varepsilon - c
 -\dashint_{D_{2^{-j+1}}} (Z_\varepsilon - c)\, dy\Big]+ \Big[p_0(x) - \dashint_{D_{2^{-j+1}}} p_0 \,dy\Big] \\
& \qquad+ \Big[\pi(\cdot/\varepsilon)\psi_{4\varepsilon}\nabla u_0
-\dashint_{D_{2^{-j+1}}} \pi(y/\varepsilon)\psi_{4\varepsilon}\nabla u_0\,dy\Big]\bigg\}  dx\Bigg|.
\end{aligned}
\end{equation}
Now if we set $c = \overline{(Z_\varepsilon)}_{D_{2^{-j+1}}}$, then the right-hand side of $\eqref{f:6.9}$ is controlled by
\begin{equation*}
\begin{aligned}
\Big(\dashint_{D_{2^{-j}}} \big|Z_\varepsilon
 &- \overline{(Z_\varepsilon)}_{D_{2^{-j+1}}}\big|^2 dx\Big)^{\frac{1}{2}}
 + \Big(\dashint_{D_{2^{-j}}}\big|p_0 - \dashint_{D_{2^{-j+1}}}
 p_0 \big|^2 dx\Big)^{\frac{1}{2}} \\
 & + \bigg|\dashint_{D_{2^{-j}}} \Big[\pi(\cdot/\varepsilon)\psi_{4\varepsilon}\nabla u_0
-\dashint_{D_{2^{-j+1}}} \pi(y/\varepsilon)\psi_{4\varepsilon}\nabla u_0 dy\Big] dx\bigg|
 =: I_1 + I_2 + I_3.
\end{aligned}
\end{equation*}
We first handle $I_1$, and it follows that
\begin{equation}\label{f:6.13}
\begin{aligned}
I_1 &\leq C\Big(\dashint_{D_{2^{-j+1}}} \big|Z_\varepsilon
- \overline{(Z_\varepsilon)}_{D_{2^{-j+1}}}\big|^2 dx\Big)^{\frac{1}{2}} \\
&\leq C\left(\frac{\varepsilon}{2^{-j}}\right)^{\lambda}
\Bigg\{\frac{1}{2^{-j}}\Big(\dashint_{D_{2^{-j}}}|u_\varepsilon|^2dx\Big)^{\frac{1}{2}}
+2^{-j}\Big(\dashint_{D_{2^{-j}}}|F|^pdx\Big)^{\frac{1}{p}}\\
&\qquad\qquad\qquad\qquad
+2^{-j}\Big(\dashint_{D_{2^{-j}}}|\nabla h|^pdx\Big)^{\frac{1}{p}}
+ \|h\|_{L^\infty(D_2)} + \|\nabla g\|_{C^{0,\eta}(\Delta_2)} \Bigg\} \\
&\leq C(\varepsilon 2^j)^{\lambda}\bigg\{\Big(\dashint_{D_{2}}|u_\varepsilon|^2dx\Big)^{\frac{1}{2}}
+\|F\|_{L^p(D_2)} + \|\nabla h\|_{L^{p}(D_2)}
+ \|h\|_{L^\infty(D_2)} + \|\nabla g\|_{C^{0,\eta}(\Delta_2)}\bigg\},
\end{aligned}
\end{equation}
where we use the Lipschitz estimate $\eqref{f:5.16}$ in the last step, as well as the fact $p>d$.
Then we proceed to study $I_2$, and it follows from the estimate $\eqref{pri:6.5}$ that
\begin{equation}\label{f:6.14}
I_2 \leq C(2^{-j})^{\rho}\bigg\{\|u_0\|_{L^2(D_2)}
 +\|F\|_{L^{p}(D_2)}
 + \|h\|_{W^{1,p}(D_2)}
 +\|\nabla g\|_{C^{0,\eta}(\Delta_2)}\bigg\}.
\end{equation}
We now investigate $I_3$. In the case of $\varepsilon \leq 2^{-j+1} <3\varepsilon$,
there is nothing to do since the term $S_\varepsilon(\psi_{4\varepsilon}\nabla u_0)$ is supported outside
$D_{2^{-j+1}}$. Hence we only deal with $I_3$ in the case of $2^{-j}\geq 3\varepsilon$.
\begin{equation*}
\begin{aligned}
I_3 &\leq \bigg|
\dashint_{D_{2^{-j}}} \big[\pi(\cdot/\varepsilon)-\widehat{\pi}\big]
\psi_{4\varepsilon}\nabla u_0 dx\bigg|
+ \bigg|\dashint_{D_{2^{-j+1}}} \big[\pi(\cdot/\varepsilon)-\widehat{\pi}]
\psi_{4\varepsilon}\nabla u_0 dx \bigg|\\
& + \big|\widehat{\pi}\big|~\bigg|
\dashint_{D_{2^{-j}}}\Big(\psi_{4\varepsilon}\nabla u_0
-\dashint_{D_{2^{-j+1}}}\psi_{4\varepsilon}\nabla u_0dy\Big) dx\bigg|
=:I_{31} + I_{32} + I_{33},
\end{aligned}
\end{equation*}
where $\widehat{\pi} = \dashint_Y \pi(y) dy$.
Note that $I_{31}$ and $I_{32}$ obey the same computations. Taking $I_{31}$ for example,
there exists $V\in H^1_{per}(Y;\mathbb{R}^d)$ such that
$\text{div}_y(V)= \pi(y) -\widehat{\pi}$ in $\mathbb{R}^d$. Thus we have
\begin{equation}\label{f:6.11}
\begin{aligned}
I_{31}
& = \bigg|\dashint_{D_{2^{-j}}} \big[\pi(\cdot/\varepsilon)-\widehat{\pi}\big]
\psi_{4\varepsilon}\nabla u_0 dx\bigg|
= \varepsilon \bigg|\dashint_{D_{2^{-j}}} \text{div}_x \big[V(x/\varepsilon)\big]
\psi_{4\varepsilon}\nabla u_0 dx\bigg| \\
& \leq \varepsilon \dashint_{D_{2^{-j}}} \Big|V(x/\varepsilon)\cdot \nabla
(\psi_{4\varepsilon}\nabla u_0) \Big|dx
+ \frac{C\varepsilon}{2^{-j}}\bigg(\dashint_{\partial(D_{2^{-j}})}
\Big|V(x/\varepsilon)
\psi_{4\varepsilon}\nabla u_0 \Big|^2dx\bigg)^{\frac{1}{2}} \\
&\leq \frac{C}{|D_{2^{-j}}|^{\frac{1}{2}}}\Big\{\big\|\nabla u_0\big\|_{L^2(D_{2^{-j}}\setminus\Sigma_{4\varepsilon})}
+ \varepsilon\big\|\nabla^2 u_0\big\|_{L^2(D_{2^{-j}}\cap\Sigma_{4\varepsilon})}\Big\}
+C\varepsilon 2^{j}\|\nabla u_0\|_{L^\infty(D_{1/2})}\\
&\leq C\Big\{(\varepsilon 2^j)^{\frac{1}{2}}
+ (\varepsilon 2^j)\Big\}
\Bigg\{
\Big(\dashint_{D_{2}}|u_0|^2 dx\Big)^{\frac{1}{2}}
+ \|F\|_{L^p(D_2)}
+ \|h\|_{W^{1,p}(D_2)}
+ \|\nabla g\|_{C^{0,\eta}(\Delta_{2})}\Bigg\},
\end{aligned}
\end{equation}
where we use Remark $\ref{remark:3.1}$ in the second inequality, and the estimates $\eqref{pri:3.15}$, $\eqref{pri:3.16}$
and $\eqref{pri:2.17}$ in the last one.
We now turn to estimate $I_{33}$. It follows that
\begin{equation}\label{f:6.12}
\begin{aligned}
I_{33} &= \big|\widehat{\pi}\big|~\bigg|
\dashint_{D_{2^{-j}}}\Big(\psi_{4\varepsilon}\nabla u_0
-\dashint_{D_{2^{-j+1}}}\psi_{4\varepsilon}\nabla u_0dy\Big) dx\bigg|
\leq C\dashint_{D_{2^{-j}}}\big|\psi_{4\varepsilon}\nabla u_0
- \nabla u_0(0)\big| dx\\
&\leq  C\dashint_{D_{2^{-j}}}\Big|\psi_{4\varepsilon}\big(\nabla u_0
-\nabla u_0(0)\big) + (\psi_{4\varepsilon} - 1)\nabla u_0(0)\Big| dx \\
&\leq C(2^{-j})^\rho\big[\nabla u_0\big]_{C^{0,\rho}(D_1)}
+ C(\varepsilon 2^j)\|\nabla u_0\|_{L^\infty(D_1)}\\
&\leq  C\Big\{(2^{-j})^{\rho}+ (\varepsilon 2^{j})\Big\}\bigg\{\|u_0\|_{L^2(D_2)}
 +\|F\|_{L^{p}(D_2)}
 + \|h\|_{W^{1,p}(D_2)}
 +\|\nabla g\|_{C^{0,\eta}(\Delta_2)}\bigg\},
\end{aligned}
\end{equation}
where we employ the estimate $\eqref{pri:2.17}$ and $\eqref{pri:2.18}$ in the last step.
Moreover, it is clear to see that the estimates $\eqref{f:6.11}$ and $\eqref{f:6.12}$ give
\begin{equation}\label{f:6.15}
\begin{aligned}
I_{3}
&\leq C\Big\{(\varepsilon 2^j)+ (\varepsilon 2^j)^{\frac{1}{2}}
+ (2^{-j})^\rho\Big\} \bigg\{\Big(\dashint_{D_2}|u_0|^2 dx\Big)^{\frac{1}{2}}
+\|F\|_{L^p(D_2)}
+\|h\|_{W^{1,p}(D_2)}
+\|\nabla g\|_{C^{0,\eta}(\Delta_2)}\bigg\}.
\end{aligned}
\end{equation}
Combining the estimates $\eqref{f:6.13}$, $\eqref{f:6.14}$ and $\eqref{f:6.15}$, we obtain that
\begin{equation*}
\begin{aligned}
\Big|\dashint_{D_r} &p_\varepsilon - \dashint_{D_1} p_\varepsilon dx\Big|
\leq \sum_{j=1}^k (\varepsilon 2^j)^{\frac{1}{2}}
\bigg\{\Big(\dashint_{D_2}|u_\varepsilon|^2 dx\Big)^{\frac{1}{2}}
+\|F\|_{L^p(D_2)} + \|h\|_{W^{1,p}(D_2)}
+ \|\nabla g\|_{C^{0,\eta}(\Delta_2)} \bigg\}\\
& +\sum_{j=1}^k C\Big\{(\varepsilon 2^j)+ (\varepsilon 2^j)^{\frac{1}{2}}
+ (2^{-j})^\rho\Big\} \bigg\{\Big(\dashint_{D_2}|u_0|^2 dx\Big)^{\frac{1}{2}}
+\|F\|_{L^p(D_2)}
+\|h\|_{W^{1,p}(D_2)}
+\|\nabla g\|_{C^{0,\eta}(\Delta_2)}\bigg\}.
\end{aligned}
\end{equation*}
Noting that $0<\varepsilon<2^{-k}$, there exists a constant $C$ independent of $k$ such that
\begin{equation}\label{f:6.10}
\begin{aligned}
\Big|\dashint_{D_r} p_\varepsilon - \dashint_{D_1} p_\varepsilon dx\Big|
&\leq C\Bigg\{\Big(\dashint_{D_2}|u_\varepsilon|^2 dx\Big)^{\frac{1}{2}}
+ \Big(\dashint_{D_2}|u_0|^2 dx\Big)^{\frac{1}{2}}+\Big(\dashint_{D_2}|F|^p dx\Big)^{\frac{1}{p}}\\
&\qquad+\Big(\dashint_{D_2}|\nabla h|^p dx\Big)^{\frac{1}{p}}
+\|h\|_{L^{\infty}(D_2)}
+\|\nabla g\|_{C^{0,\eta}(\Delta_2)}\Bigg\}.
\end{aligned}
\end{equation}
In fact, we know that $u_\varepsilon\to u_0$ strongly in $L^2(D_2;\mathbb{R}^d)$ according to
the construction of $(u_0,p_0)$ in Lemmas $\ref{lemma:5.4}$ and $\ref{lemma:6.3}$. Hence, it is not hard
to see that
\begin{equation*}
\Big(\dashint_{D_2}|u_0|^2 dx\Big)^{\frac{1}{2}}
\leq C\Big(\dashint_{D_2}|u_\varepsilon|^2 dx\Big)^{\frac{1}{2}}
\end{equation*}
Put this inequality into $\eqref{f:6.10}$, and we finally derive the desired estimate.
\qed

\begin{flushleft}
\textbf{Proof of Theorem $\ref{thm:1.1}$.}
The desired estimate
$\eqref{pri:1.1}$ directly follows from
Theorems $\ref{thm:5.2}$ and $\ref{thm:6.2}$.
\qed
\end{flushleft}

\begin{flushleft}
\textbf{Proof of Theorem $\ref{thm:1.0}$.}
The estimate
$\eqref{pri:1.0}$ follows from
Theorem $\ref{thm:1.1}$ and \cite[Theorem 1.1]{SGZWS},
and we are done.
\qed
\end{flushleft}

\begin{center}
\textbf{Acknowledgements}
\end{center}

Both of the authors want to express their sincere appreciation to Professor Zhongwei Shen for his constant guidance and encouragement.
The second author was supported by the National Natural Science Foundation of China (Grant NO.11471147).


\begin{thebibliography}{000}
\bibitem{SACS} S. Armstrong, C. Smart,
Quantitative stochastic homogenization of convex integral functionals, Ann. Sci. \'Ec.
Norm. Sup\'er., (to appear).

\bibitem{SZ} S. Armstrong, Z. Shen, Lipschitz estimates in almost-periodic homogenization,
Comm. Pure Appl. Math.  69(2016), no.10, 1882-1923.


\bibitem{MAFHL} M. Avellaneda and F. Lin, Compactness methods in the theory of homogenization, Comm. Pure Appl. Math. 40(1987), no.6, 803-847.

\bibitem{ABJLGP} A. Bensoussan, J.-L. Lions, and G.C. Papanicolaou, Asympotic Analysis for Periodic Structures, Studies in Mathematics and its Applications, North Holland, 1978.





\bibitem{MGMG} M. Giaquinta, G. Modica, Non-linear systems of the type of the stationary Navier-Stokes system,
J. Reine Angew. Math. 330(1982), 173-214.

\bibitem{GPGCG} G.P. Galdi, C.G. Simader, Existence, uniqueness and $L^q$ estimates for the Stokes problem in an exterior domain, Arch. Rational Mech. Anal. 112(1990), 291-318.

\bibitem{GZ} J. Geng, Z. Shen, Uniform regularity estimates in parabolic homogenization, Indiana Univ. Math. J.  64 (2015), no.3, 697-733.

\bibitem{GZS} J. Geng, Z. Shen, Convergence rates in parabolic homogenization with time-dependent coefficients, J. Funct. Anal. (in press).

\bibitem{GZS1} J. Geng, Z. Shen, L. Song, Boundary Korn inequality and Neumann problems in homogenization of systems of elasticity, arXiv:1608.07736v1(2016).

\bibitem{G} S. Gu, Convergence rates in homogenization of Stokes systems,
J. Differential Equations 260(2016), no.7, 5796-5815.

\bibitem{G1} S. Gu, Convergence Rates of Neumann problems for Stokes Systems, arXiv:1512.08285.

\bibitem{G2} S. Gu,Homogenization of Stokes Systems with Periodic Coefficients, Thesis (Ph.D.) University of Kentucky, 2016.

\bibitem{SGZWS} S. Gu, Z. Shen,
Homogenization of Stokes systems and uniform regularity estimates,
SIAM J. Math. Anal. 47(2015),  no.5, 4025-4057.

\bibitem{SZW4} C.E. Kenig, F. Lin, Z. Shen, Homogenization of elliptic systems with Neumann boundary conditions, J. Amer. Math. Soc. 26(2013), no.4, 901-937.

\bibitem{SZW2} C.E. Kenig, F. Lin, Z. Shen, Convergence rates in $L^2$ for elliptic homogenization problems, Arch. Ration. Mech. Anal. 203(2012), no.3, 1009-1036.

\bibitem{OAL} O.A. Ladyzhenskaya, The Mathematical Theory of Viscous Incompressible Flow,
Revised English edition, Gordon and Breach Science Publishers, New York-London, 1963.

\bibitem{SZW12} Z. Shen, Boundary estimates in elliptic homogenization, arXiv:1505.00694v1(2015).
\bibitem{SZW20} Z. Shen, J. Zhuge, Boundary layers in periodic homogenization of Neumann problems,
arXiv:1610.05273v2(2016).

\bibitem{SZW21} Z. Shen, J. Zhuge, Approximate correctors and convergence rates in almost-periodic homogenization, arXiv:1603.03139v2(2016).

\bibitem{TS2} T. Suslina, Homogenization of the Dirichlet problem for elliptic systems: $L_2$-operator error estimates, Mathematika 59(2013), no.2, 463-476.

\bibitem{TS} T. Suslina, Homogenization of the Neumann problem for elliptic systems with periodic coefficients, SIAM J. Math. Anal. 45(2013), no.6, 3453-3493.


\bibitem{QX3} Q. Xu, Convergence reates and $W^{1,p}$ estimates in homogenization theory of Stokes systems in Lipschitz domains, arXiv:1609.00122v1 (2016).

\bibitem{QX2} Q. Xu, Convergence rates for general elliptic homogenization
problems in Lipschitz domains, SIAM J. Math. Anal. 48(2016), no.6, 3742-3788.

\bibitem{QXS1} Q. Xu, Uniform regularity estimates in homogenization theory of elliptic
systems with lower order terms on the Neumann boundary problem, J. Differential Equations 261(2016), no.8, 4368-4423.

\bibitem{ZVVPSE} V.V. Zhikov, S.E. Pastukhova, On operator estimates for some problems in homogenization theory, Russ. J. Math. Phys. 12(2005), no. 4, 515-524.

\bibitem{ZG1}  J. Zhuge, Uniform boundary regularity in almost-periodic homogenization. J. Differential Equations 262(2017), no.1, 418-453.
\end{thebibliography}
\end{document}